\documentclass[12pt]{amsart}
\usepackage{amssymb}
\usepackage{amsmath}
\usepackage{amsthm}
\usepackage{framed}
\usepackage{xcolor}
\usepackage{mathtools}
\usepackage{mathrsfs}
\usepackage{tikz-cd, amssymb}
\usetikzlibrary{decorations.pathmorphing}
\usetikzlibrary{shapes.geometric}
\usepackage{ upgreek }
\usepackage{bbm}
\usepackage{enumitem}
\usepackage{stmaryrd}
\usepackage{blindtext}
\usepackage[hidelinks]{hyperref}
\usepackage[margin=2.5cm]{geometry}
\usepackage{verbatim}
\usepackage{multicol}
\usepackage{cancel}
\usepackage{tabularray}

\definecolor{myred}{RGB}{250,170,85}
\definecolor{mygreen}{RGB}{155,250,100}
\definecolor{myyellow}{RGB}{245,215,25}
\definecolor{mypurple}{RGB}{170,85,250}

\makeatletter
\newcommand\RedeclareMathOperator{%
  \@ifstar{\def\rmo@s{m}\rmo@redeclare}{\def\rmo@s{o}\rmo@redeclare}%
}
\newcommand\rmo@redeclare[2]{%
  \begingroup \escapechar\m@ne\xdef\@gtempa{{\string#1}}\endgroup
  \expandafter\@ifundefined\@gtempa
     {\@latex@error{\noexpand#1undefined}\@ehc}%
     \relax
  \expandafter\rmo@declmathop\rmo@s{#1}{#2}}
\newcommand\rmo@declmathop[3]{%
  \DeclareRobustCommand{#2}{\qopname\newmcodes@#1{#3}}%
}
\@onlypreamble\RedeclareMathOperator
\makeatother

\DeclareMathOperator{\Hom}{\mathrm{Hom}}
\DeclareMathOperator{\CAlg}{\mathrm{CAlg}}
\DeclareMathOperator{\RR}{\mathbb{R}}
\DeclareMathOperator{\CC}{\mathbb{C}}
\DeclareMathOperator{\ZZ}{\mathbb{Z}}
\DeclareMathOperator{\QQ}{\mathbb{Q}}
\DeclareMathOperator{\Z2}{\mathbb{Z}/2}

\DeclareMathOperator{\CP}{\mathbb{CP}}
\DeclareMathOperator{\HH}{\mathcal{H}}
\DeclareMathOperator{\BB}{\mathcal{B}}
\DeclareMathOperator{\KK}{\mathcal{K}}
\DeclareMathOperator{\II}{\mathcal{I}}
\DeclareMathOperator{\KO}{\mathrm{KO}}
\DeclareMathOperator{\KU}{\mathrm{KU}}
\DeclareMathOperator{\KR}{\mathrm{KR}}
\DeclareMathOperator{\Cl}{\mathrm{Cl}}
\DeclareMathOperator{\CCl}{\CC\!\mathrm{l}}

\DeclareMathOperator{\id}{\mathrm{id}}

\DeclareMathOperator{\Spin}{\mathrm{Spin}}
\DeclareMathOperator{\Pin}{\mathrm{Pin}}
\DeclareMathOperator{\Aut}{\mathrm{Aut}}

\DeclareMathOperator{\BSpin}{\mathrm{BSpin}}
\DeclareMathOperator{\ESpin}{\mathrm{ESpin}}
\DeclareMathOperator{\MSpin}{\mathrm{MSpin}}
\DeclareMathOperator{\MSO}{\mathrm{MSO}}
\DeclareMathOperator{\MO}{\mathrm{MO}}
\DeclareMathOperator{\U}{\mathrm{U}}
\RedeclareMathOperator{\O}{\mathrm{O}}
\DeclareMathOperator{\PU}{\mathrm{PU}}
\DeclareMathOperator{\ev}{\mathrm{even}}

\DeclareMathOperator{\BU}{\mathrm{BU}}
\DeclareMathOperator{\MU}{\mathrm{MU}}
\DeclareMathOperator{\s}{\mathfrak{s}}
\DeclareMathOperator{\colim}{\mathrm{colim}}
\newcommand{\Sp}{\textnormal{Sp}}

\newcommand{\fc}{\textnormal{fc}}
\newcommand{\gr}{\textnormal{gr}}

\newcommand{\T}{\mathcal{T}}

\theoremstyle{definition}
\newtheorem{theorem}{Theorem}[section]
\newtheorem{lemma}[theorem]{Lemma}

\newtheorem{corollary}[theorem]{Corollary}

\newtheorem{proposition}[theorem]{Proposition}

\newtheorem{example}[theorem]{Example}

\numberwithin{subcase}{case}
\newtheorem{remark}[theorem]{Remark}

\newtheorem{definition}[theorem]{Definition}

\title{Real spin bordism and orientations of topological $\mathrm{K}$-theory}
\author{Zachary Halladay}
\address{Department of Mathematics, University of Illinois at Urbana-Champaign, Urbana, IL, USA}
\email{zah2@illinois.edu} 
\author{Yigal Kamel}
\address{Department of Mathematics, University of Illinois at Urbana-Champaign, Urbana, IL, USA}
\email{ykamel2@illinois.edu}

\begin{document}

\begin{abstract}
We construct a commutative orthogonal $C_2$-ring spectrum, $\mathrm{MSpin}^c_{\mathbb{R}}$, along with a $C_2$-$E_{\infty}$-orientation $\mathrm{MSpin}^c_{\mathbb{R}} \to \mathrm{KU}_{\mathbb{R}}$ of Atiyah's Real K-theory. Further, we define $E_{\infty}$-maps $\mathrm{MSpin} \to (\mathrm{MSpin}^c_{\mathbb{R}})^{C_2}$ and $\mathrm{MU}_{\mathbb{R}} \to \mathrm{MSpin}^c_{\mathbb{R}}$, which are used to recover the three well-known orientations of topological $\mathrm{K}$-theory, $\mathrm{MSpin}^c \to \mathrm{KU}$, $\mathrm{MSpin} \to \mathrm{KO}$, and $\mathrm{MU}_{\mathbb{R}} \to \mathrm{KU}_{\mathbb{R}}$, from the map $\mathrm{MSpin}^c_{\mathbb{R}} \to \mathrm{KU}_{\mathbb{R}}$. We also show that the integrality of the $\hat{A}$-genus on spin manifolds provides an obstruction for the fixed points $(\mathrm{MSpin}^c_{\mathbb{R}})^{C_2}$ to be equivalent to $\mathrm{MSpin}$, using the Mackey functor structure of $\underline{\pi}_*\mathrm{MSpin}^c_{\mathbb{R}}$. In particular, the usual map $\mathrm{MSpin} \to \mathrm{MSpin}^c$ does not arise as the inclusion of fixed points for any $C_2$-$E_{\infty}$-ring spectrum.
\end{abstract}

\maketitle

\tableofcontents

\section{Introduction}\label{Introduction}

The main result of this paper refines all of the standard orientations of topological $\mathrm{K}$-theory to a single \textit{Real spin orientation}, $\MSpin^c_{\RR} \to \KU_{\RR} = \KR$, of Atiyah's Real $\mathrm{K}$-theory.

\begin{theorem}[Propositions \ref{realspinbordism}, \ref{underlying.spinc}, \ref{spin.fixedpoints}, \ref{Real.orientation}, and \ref{Real.ABS}]\label{maintheorem}
There exists a map of $C_2$-$E_{\infty}$-ring spectra, $\MSpin^c_{\RR} \to \KU_{\RR}$, satisfying the following properties. 
\begin{enumerate}
    \item The underlying spectrum of $\MSpin^c_{\RR}$ is $\MSpin^c$.
    \item There exists an $E_{\infty}$-map $\MSpin \to (\MSpin^c_{\RR})^{C_2}$.
    \item There exists a $C_2$-$E_{\infty}$-map $\MU_{\RR} \to \MSpin^c_{\RR}$.  
\end{enumerate}
The three standard orientations, $\MSpin^c \to \KU$, $\MSpin \to \KO$, and $\MU_{\RR} \to \KU_{\RR}$, are  recovered as $E_{\infty}$-maps from the map $\MSpin^c_{\RR} \to \KU_{\RR}$. 
\end{theorem}

In analogy with Real bordism and Real K-theory, we use the name \textit{Real spin bordism} to refer to the $C_2$-spectrum $\MSpin^c_{\RR}$. The Real spin orientation $\MSpin^c_{\RR} \to \KU_{\RR}$ recovers the standard orientations of K-theory as follows. The Real orientation of $\KU_{\RR}$ is obtained by precomposing with $\MU_{\RR} \to \MSpin^c_{\RR}$, the spin${}^c$ orientation of $\KU$ is obtained by taking underlying spectra, and the spin orientation of KO is obtained as the composite $$\MSpin \to (\MSpin^c_{\RR})^{C_2} \to \KU_{\RR}^{C_2} \simeq \KO.$$ 

We construct $\MSpin^c_{\RR}$ as follows. Let $\RR^{p,q}$ be the $C_2$-representation $\RR^p \oplus \RR^q  \sigma$, where $\sigma$ is the sign representation. Then the complex Clifford algebra $\CCl(\RR^{p,q})$ inherits a corresponding conjugate-linear $C_2$-action by the universal property of Clifford algebras. This action restricts to $\Spin^c(p,q) \subset \CCl(\RR^{p,q})$, and makes the usual representation, $\Spin^c(p,q) \times \RR^{p,q} \to \RR^{p,q}$,  $C_2$-equivariant. Using an appropriate model of $\ESpin^c(p,q)$, this induces a $C_2$-action on the Thom space, $\MSpin^c(p,q)$, of the universal $\Spin^c(p,q)$-vector bundle with fiber $\RR^{p,q}$. We define $\MSpin^c_{\RR}$ as an orthogonal $C_2$-spectrum whose value on the $C_2$-representation $\RR^{p,q}$ is the $C_2$-space $\MSpin^c(p,q)$.

\begin{remark}
Complex conjugation on $\U(1)$ induces a $C_2$-action on $\MSpin^c(n)$, which describes the underlying spectrum of $\MSpin^c_{\RR}$ with its induced $C_2$-action. This spectrum with $C_2$-action can be promoted to a genuine $C_2$-spectrum, $(\MSpin^c_{\RR})^h$, called the \textit{cofree completion} of $\MSpin^c_{\RR}$. Using the completion map $\MSpin^c_{\RR} \to (\MSpin^c_{\RR})^h$ and the fact that $\KU_{\RR} \simeq \KU_{\RR}^h$, all of the results in this paper apply to $(\MSpin^c_{\RR})^h$ as well. The question of whether or not $\MSpin^c_{\RR}$ is equivalent to $(\MSpin^c_{\RR})^h$ amounts to knowing whether or not the fixed points $(\MSpin^c_{\RR})^{C_2}$ is equivalent to the homotopy fixed points $(\MSpin^c_{\RR})^{hC_2} = ((\MSpin^c_{\RR})^h)^{C_2}$, which we do not investigate in this paper. 
\end{remark}

Given the facts that $\Spin^c(n)^{C_2} = \Spin(n)$ and $\KU_{\RR}^{C_2} = \KU_{\RR}^{hC_2} = \KO$, one might expect that the $C_2$-(homotopy) fixed points of $\MSpin^c_{\RR}$ is equivalent to $\MSpin$. The following theorem says that no such equivalence can cover $\MSO$.

\begin{theorem}\label{intro.counterspin}
    There does not exist a genuine $C_2$-spectrum $E$ satisfying all of the following conditions. 
    \begin{enumerate}
        \item The homotopy groups of the underlying spectrum are $E^e_* \cong \MSpin^c_*$;
        \item The homotopy groups of the $C_2$-fixed points are $E^{C_2}_* \cong \MSpin_*$;
        \item The $C_2$-action on $E^e_*$ is via a ring homomorphism on $\MSpin^c_*$;
        \item The diagram 
        $$
        \begin{tikzcd}
            \MSpin_*  \arrow[rr, "\text{res}"] \arrow[dr, "u_*"'] & & \MSpin^c_* \arrow[dl, "u_*^c"] \\
            & \MSO_* &
        \end{tikzcd}
        $$
        commutes, where $u_*$ and $u_*^c$ are the forgetful maps that take the underlying oriented bordism class, and $\text{res} \colon \MSpin_* \cong E^{C_2}_* \to E^e_* \cong \MSpin^c_*$ is induced by inclusion of fixed points.
    \end{enumerate}
\end{theorem}

Since the maps $\MSpin \to (\MSpin^c_{\RR})^{C_2}$ and $\MSpin \to (\MSpin^c_{\RR})^{hC_2}$ from Theorem \ref{maintheorem} factor the usual map $\MSpin \to \MSpin^c$, Theorem \ref{intro.counterspin} implies that neither map is an equivalence. We do not know whether or not either  $(\MSpin^c_{\RR})^{C_2}$ or $(\MSpin^c_{\RR})^{hC_2}$ is equivalent to MSpin, but since Theorem \ref{intro.counterspin} implies that the expected comparison map cannot be an equivalence, there is no reason to expect that such an equivalence exists.

\begin{remark}
    In light of the maps $\MU_{\RR} \to \MSpin^c_{\RR} \to \KU_{\RR}$, it is natural to compare $\MSpin^c_{\RR}$ to what is known about the various types of fixed points of $\MU_{\RR}$ and $\KU_{\RR}$. In particular, in the case of Real bordism, the $C_2$-(homotopy) fixed points $\MU_{\RR}^{C_2} = \MU_{\RR}^{hC_2}$ is not equivalent to MO, but the \textit{geometric fixed points}, $\Phi^{C_2}\MU_{\RR}$, is. Despite this analogy, we know that $\Phi^{C_2}\MSpin^c_{\RR} \not\simeq \MSpin$, since the functor $\Phi^{C_2}$ is lax monoidal, and there does not exist a ring map from MO to MSpin. We collect this information in a table in Figure \ref{figure.fixedpoints}.
\end{remark}

\begin{figure}[h]
\centering
\begin{tblr}{
colspec = {| c | c | c | c |}, 
cell{2}{2} = {myred!40},
cell{3}{2} = {myred!40},
cell{4}{2} = {mygreen!40},
cell{2}{3} = {mypurple!32},
cell{3}{3} = {mypurple!32},
cell{4}{3} = {myred!40},
cell{2}{4} = {mygreen!40},
cell{3}{4} = {mygreen!40},
cell{4}{4} = {myred!40} 
} \hline
 type of $C_2$-fixed points & $\MU_{\RR}$ & $\MSpin^c_{\RR}$ & $\KU_{\RR}$ \\ \hline 
 genuine $\;\;(\;\;)^{C_2}$ & $\cancel{\MO}$ & ? & KO \\  \hline
 homotopy $\;\;(\;\;)^{hC_2}$ & $\cancel{\MO}$  & ? & KO \\ \hline
 geometric $\;\;\Phi^{C_2}(\;\;)$ & $\MO$ & $\cancel{\MSpin}$  & $\cancel{\KO}$ \\ \hline
\end{tblr}
\caption{Comparing the various types of fixed points of Real bordism, Real spin bordism, and Real $\mathrm{K}$-theory}\label{figure.fixedpoints}
\end{figure}

\begin{remark}
Even though $\MU_{\RR}$ and $\MSpin^c_{\RR}$ are both Thom spectra built using the groups $\U(n)$ and $\Spin^c(n)$ with $\U(n)^{C_2} = \O(n)$ and $\Spin^c(n)^{C_2} = \Spin(n)$, the analogy between them breaks down for the following reason. The $C_2$-action on $\U(n)$ determines the action on the $(\RR^{n,n})$th space of the Real bordism spectrum, while the $C_2$-action on $\Spin^c(n)$ determines the action on the $(\RR^n)$th space of the Real spin bordism spectrum. The group that corresponds to the $(\RR^{n,n})$th space of $\MSpin^c_{\RR}$ is $\Spin^c(n,n)$, and $\Spin^c(n,n)^{C_2} \not\cong \Spin(n)$.     
\end{remark}

Lastly, we point out that the cohomological Thom classes that correspond to the Real spin orientation of $\KU_{\RR}$ constructed in this paper were originally constructed by Atiyah in the form of the following generalized Bott periodicity theorem.

\begin{theorem}[Atiyah \cite{Atiyah}, Theorem 6.3]\label{atiyah.Thom}
Let $G \to \Spin^c(p,q)$ be a $C_2$-equivariant homomorphism of groups with $C_2$-action, and suppose $p \equiv q \mod 8$. Then there is a class $u \in \KU_{\RR, G}(\RR^{p,q})$, such that multiplication by $u$ induces an isomorphism, 
$$
\KU_{\RR,G}(X) \cong \KU_{\RR, G}(\RR^{p,q} \times X).
$$
\end{theorem}

As in the case of Theorem \ref{maintheorem}, this theorem contains all of the classical Thom isomorphisms as special cases. For example, letting $G = \Spin(8n)$ with trivial $C_2$-action, $(p,q) = (8n,0)$, and taking $X$ to be the total space of a principal $\Spin(8n)$-bundle, this recovers the Atiyah--Bott--Shapiro Thom class in $\KO$-theory for the spin vector bundle associated to $X$ \cite{ABS}. Similarly, the $\KU$-Thom classes for spin${}^c$ vector bundles are obtained from the identity map $\Spin^c(8n) = \Spin^c(8n,0)$, and the $\KU_{\RR}$-Thom classes for Real vector bundles are obtained from the inclusion $\U(n) \to \Spin^c(n,n)$.

\subsection{Overview of the paper} In Section \ref{sec.prelim}, we establish some terminology about Real structures and review facts that we need about orthogonal $C_2$-spectra. In Section \ref{sec.realspin}, we construct $\MSpin^c_{\RR}$ as a commutative monoid object in the category of orthogonal $C_2$-spectra. The spaces we use to model $\MSpin^c_{\RR}$ follow the definition of $\MSpin^c$ in \cite{Joachim}, except that we introduce (different) $C_2$-actions. In section \ref{sec.fixed}, we define the map $\MSpin \to (\MSpin^c_{\RR})^{C_2}$ and provide a general obstruction for such a map to be an equivalence by proving Theorem \ref{intro.counterspin}. This obstruction uses the integrality of the $\hat{A}$-genus on spin manifolds, together with formal properties that the homotopy Mackey functor $\underline{\pi}_*\MSpin^c_{\RR}$ must satisfy. In section \ref{sec.MR}, we construct the map $\MU_{\RR} \to \MSpin^c_{\RR}$ in the category of Real spectra (\cite{HHR}, \cite{hill_hopkins_ravenel_2021}). To do this, we use a model of $\MU_{\RR}$ that is mostly derived from our model of $\MSpin^c_{\RR}$ at the regular representations. The desired map then follows from the fact that $\U(n) \to \Spin^c(n,n)$ is $C_2$-equivariant. Most of the work in section \ref{sec.MR} lies in showing that our definition of $\MU_{\RR}$ is actually equivalent to Real bordism, which we do by describing the $C_2$-fixed points of the Real vector bundle associated to a Real free $\U(n)$-space (Lemma \ref{fixedpoints.MU}). In section \ref{sec.KR}, we construct the map $\MSpin^c_{\RR} \to \KU_{\RR}$ in the category of orthogonal $C_2$-spectra. Our model of $\KU_{\RR}$ follows the definition of $\KU$ in \cite{Joachim} in terms of spaces of graded $*$-homomorphisms of $C^*$-algebras, except that we endow all relevant $C^*$-algebras with Real structures. We show that our model of $\KU_{\RR}$ represents Atiyah's Real K-theory by showing that the associated cohomology theory restricted to spaces with trivial $C_2$-action is equivalent to KO.

\subsection{Acknowledgements} There are several people we would like to thank for helping this work come to fruition: Arun Debray, Cameron Krulewski, Natalia Pacheco-Tallaj, and Luuk Stehouwer,  for helping prove an earlier version of Corollary \ref{cor.counterspin}, which motivated our use of a genus in the proof of Theorem \ref{intro.counterspin}; Hassan Abdallah, Mike Hill, Michael Joachim, Kiran Luecke, and Charles Rezk, for helpful discussions relating to various aspects of this paper; and Dan Berwick-Evans, Connor Grady, Doron Grossman-Naples, Cameron Krulewski, Fredrick Mooers, and Vesna Stojanoska, for helpful feedback on an earlier draft. Finally, we would like to specially thank our advisors, Vesna Stojanoska and Dan Berwick-Evans, respectively, for their many hours of guidance both before and during the writing of this paper. The first author was partially funded by NSF grant DMS-2304797. The second author was partially funded by NSF grant DMS-2205835.

\section{Preliminaries}\label{sec.prelim}

First, we fix notation for the categories of spaces and $G$-spaces that we use in this paper. Let $\T$  be the category of pointed, compactly generated, weak Hausdorff spaces enriched over itself, and let $\T^{G}$ be the $\T$-enriched category of $G$-objects in $\T$ with equivariant maps. Lastly, let $\T_{G}$ be the $\T^{G}$-enriched category with the same objects as $\T^{G}$, but with $\T_{G}(X,Y) = \T(X,Y)$ which is made into a $G$-space by conjugation. Then we have $\T^{G}(X,Y) = \T_{G}(X,Y)^{G}$. 

\subsection{Real structures}\label{sec.Real}

This section introduces what the word Real refers to in various contexts. In short, \textit{Real} will always mean ``equipped with a $C_2$-action" when applied to objects, and ``respects $C_2$-actions" when applied to morphisms. We use the word Real to restrict to the context in which $\CC$ is always endowed with its fixed $C_2$-action given by complex conjugation. For example, this forces Real structures on complex vector spaces to be conjugate-linear. We often denote a Real structure (or $C_2$-action) by $x \mapsto \bar{x}$, when it is clear which action we are referring to.

\begin{definition}
 The category of \textit{Real groups} is the category of group objects in $\T^{C_2}$. If $G$ is a Real group, a \textit{Real $G$-space} is a $G$-module object in $\T^{C_2}$. We will also call the objects and morphisms of $\T^{C_2}$ \textit{Real spaces} and \textit{Real maps}, respectively. 
\end{definition}

\begin{remark}
Given a Real group $G$, a Real $G$-space is the same data as an object of $\T^{C_2 \ltimes G}$, where $C_2 \ltimes G$ is formed from the Real structure on $G$. 
\end{remark}

The following definition is more general than is needed here; we include it to indicate how Real linear algebra over the complex numbers fits into a framework similar to the one above. 

\begin{definition}
Let $\mathrm{Ab}^{C_2}$ be the symmetric monoidal category of $C_2$-objects in abelian groups under tensor product. A \textit{Real ring} is a monoid object in $\mathrm{Ab}^{C_2}$. If $R$ is a Real ring, a \textit{Real $R$-module} is an $R$-module object in $\mathrm{Ab}^{C_2}$. 
\end{definition}

\begin{example}\label{ex.C.real}
Complex conjugation gives $\CC$ the structure of a Real ring. With this Real structure, a Real $\CC$-module is a complex vector space together with a complex conjugate-linear involution. 
\end{example}

From this point on, the term \textit{Real vector space} will mean Real $\CC$-module, where $\CC$ is given the fixed real structure of Example \ref{ex.C.real}. Similarly, in any $\CC$-linear setting, such as $C^*$-algebras, $\CC$ is given this fixed Real structure. 

\begin{remark}\label{real.hom}
Let $X, Y \in \T_{C_2}$. Recall that $C_2$ acts on $\T_{C_2}(X,Y)$ by conjugation, $\overline{f}(x) = \overline{f(\overline{x})}$. This turns $\Aut_{\T_{C_2}}(X)$ into a Real group and makes $X$ into a Real $\Aut_{\T_{C_2}}(X)$-space. In this paper, we often use conjugation to define Real structures on spaces of specific kinds of functions between Real spaces. In each case, the resulting action is well-defined. The main observation is that in $\CC$-linear contexts, conjugation by $C_2$ respects $\CC$-linearity while complex conjugation of values does not.  
\end{remark}

\begin{proposition}\label{real.tensor}
Let $V,W$ be Real vector spaces. Then $V \otimes W$ is a Real vector space via $\overline{v \otimes w} = \overline{v}\otimes \overline{w}$. If $V,W$ are also Real $G$-spaces, on which $G$ acts by linear maps, then $V\otimes W$ is a Real $G$-space via $g\cdot(v\otimes w) = (gv)\otimes(gw)$.
\end{proposition}
\begin{proof}
We have $\overline{g \cdot (v \otimes w)} = \overline{gv} \otimes \overline{gw} = \overline{g} \: \overline{v} \otimes \overline{g} \: \overline{w} = \overline{g} \cdot (\overline{v \otimes w})$.  
\end{proof}

\begin{proposition}\label{real.tensor.C}
Let $V$ be a Real vector space. Then $V \cong V^{C_2} \otimes \CC$ as Real vector spaces, where the Real structure on $V^{C_2} \otimes \CC$ is given by complex conjugation on the second factor.
\end{proposition}
\begin{proof}
Define $T:V \to V^{C_2} \otimes \CC$ by $T(v) = \frac{1}{2}(v + \overline{v}) \otimes 1 - (\frac{i}{2}(v - \overline{v})) \otimes i$. First, note that $T$ is well-defined, since $(v + \overline{v}), i(v - \overline{v}) \in V^{C_2}$. Next, $T$ is Real, since 
$$
\begin{aligned}
\overline{T(v)} &= \frac{1}{2}(v + \overline{v}) \otimes 1 - \bigg(\frac{i}{2}(v - \overline{v})\bigg) \otimes \overline{i} \\
&= \frac{1}{2}(v + \overline{v}) \otimes 1 + \bigg(\frac{i}{2}(v - \overline{v})\bigg) \otimes i \\
&= \frac{1}{2}(\overline{v} + v) \otimes 1 - (\frac{i}{2}(\overline{v} - v)) \otimes i = T(\overline{v}).
\end{aligned}
$$
Lastly, $T$ is an isomorphism, since $T^{-1}: V^{C_2} \otimes \CC \to V$ is given by $T^{-1}(v_1 \otimes 1 + v_2 \otimes i) = v_1 + iv_2$.
\end{proof}

The following proposition collects various routine facts that we use in this paper. 

\begin{proposition}\label{real.props}
Let $X$ be a Real $G$-space, let $Y$ be a Real space, let $Z$ be a Real right $G$-space, and let $\psi \colon H \to G$ be a Real homomorphism. 
\begin{enumerate}
    \item The action homomorphism $\varphi \colon G \to \Aut_{\T_{C_2}}(X)$ is Real.
    \item The Real space $\T_{C_2}(X,Y)$ is a Real $G$-space via $g\cdot f = f\circ g^{-1}$. 
    \item Restriction by $\psi$ makes $X$ into a Real $H$-space.
    \item The space $Z \times_G X$ is a Real space via $\overline{[z,x]} = [\overline{z},\overline{x}]$.
\end{enumerate}
\end{proposition}
\begin{comment}
\begin{proof}
The proofs are routine; for example, we include a proof of (3). We have $$\overline{h\cdot x} = \overline{\psi(h)\cdot x} = \overline{\psi(h)}\cdot\overline{x} = \psi(\overline{h})\cdot\overline{x} = \overline{h}\cdot\overline{x}.$$
\end{proof}
\end{comment}

\subsection{Equivariant orthogonal spectra}
In this section, we recall some facts about $C_2$-spectra and fix our choice of model. For more details, we refer the reader to Hill, Hopkins, and Ravenel \cite{hill_hopkins_ravenel_2021} and \cite{HHR}, or Mandell and May \cite{K-theory/0408}. Throughout, we work in the genuine equivariant context. Our primary choice of model for this is the category of orthogonal $C_2$-spectra. 

\vspace{3mm}

Given finite dimensional inner product spaces $V, W$, let $\O(V,W)$ be the space of linear isometric embeddings $V \to W$, and let $W-V$ be the vector bundle on $\O(V,W)$ whose fiber over $\iota : V \to W$ is the orthogonal complement $W - \iota V$. When $V$ and $W$ are orthogonal $C_2$-representations, $\O(V,W)$ and $W-V$ inherit compatible $C_2$-actions. Let $\II_{C_2}$ be the $\T^{C_2}$-enriched category whose objects are finite dimensional orthogonal $C_2$-representation and whose $C_2$-space of morphisms is the Thom space
\[
\II_{C_2}(V, W) = \text{Thom}(\O(V,W), W - V).
\]
Composition is then given by applying the Thom construction to the vector bundle map
\[\begin{tikzcd}[ampersand replacement=\&]
	(V_3-V_2) \times (V_2 - V_1) \& V_3-V_1 \\
	\O(V_2, V_3) \times \O(V_1, V_2) \& \O(V_1, V_3).
	\arrow[from=1-1, to=1-2]
	\arrow[from=1-1, to=2-1]
	\arrow[from=2-1, to=2-2]
	\arrow[from=1-2, to=2-2]
\end{tikzcd}\]

\begin{definition}
An \textit{orthogonal $C_2$-spectrum} is a $\T^{C_2}$-enriched functor
\[
X \colon \II_{C_2} \to \T_{C_2}.
\]
\end{definition}

Let $\Sp^{C_2}$ be the category of orthogonal $C_2$-spectra and equivariant enriched natural transformations between them.
There is an adjunction,
\[
\Sigma_{C_2}^{\infty} \colon \T^{C_2} \rightleftarrows \Sp^{C_2} \colon \Omega_{C_2}^{\infty},
\]
where $\Sigma_{C_2}^{\infty}X$ is the orthogonal $C_2$-spectrum defined by $V \mapsto \Sigma^VX$. We let $S_{C_2}$ denote the $C_2$-sphere spectrum $\Sigma_{C_2}^{\infty}S^0$. Let $\RR \in \II_{C_2}$ denote the trivial 1-dimensional representation,  $\sigma \in \II_{C_2}$ be the 1-dimensional sign representation, and $\rho = \RR \oplus \: \sigma$ be the regular $C_2$-representation.

\begin{definition}
Given an orthogonal $C_2$-spectrum $X$, a subgroup $H \subset C_2$, and $V \in \II_{C_2}$, define the \textit{$V$th-stable $H$ homotopy group} to be
\[
\pi_V^H X = \underset{n}{\text{colim}} \: \pi_V^H\Omega^{n\rho}X(n\rho).
\]
\end{definition}
This definition can be extended to virtual representations by defining
\[
\pi_{V-W}^H X = \underset{n}{\text{colim}} \: \pi_V^H\Omega^{n\rho}X(n\rho + W).
\]
Genuine equivariance endows these homotopy groups with the additional structure of a Mackey functor $\underline{\pi}_{V-W}X = \underline{\pi}_{V-W}^{(\;\;)}X$, as $H$ varies.

\begin{definition}\label{Mackey}
    A $C_2$-\textit{Mackey Functor}, $M$, consists of Abelian groups $M(C_2)$ and $M(e)$, together with group homomorphisms
    \vspace{-5mm}
    \begin{multicols}{2}
  \begin{equation}\label{Mackey.diagram}
  \begin{tikzcd}
        M(C_2) \arrow[bend right=35,swap]{d}{\text{res}} \\
        M(e) \arrow[bend right=35,swap]{u}{\text{tr}} \arrow[out=240,in=300,loop,swap, "\overline{(\;\;)}"]
    \end{tikzcd}
  \end{equation} 
  \break
  satisfying:
    \begin{enumerate}
        \item $\overline{\text{res}(x)} = \text{res}(x)$;
        \item $\text{tr}(\overline{y}) =\text{tr}(y)$;
        \item $\overline{\overline{y}} = y$;
        \item $\text{res}(\text{tr}(y)) = y + \overline{y}$.
    \end{enumerate}
\end{multicols}
An isomorphism of $C_2$-Mackey functors is a map of diagrams \eqref{Mackey.diagram} such that each component is an isomorphism. 
\end{definition}

\begin{proposition}
    The category $\Sp^{C_2}$ together with weak equivalences given by $\underline{\pi}_*$-isomorphisms is a homotopical category\footnote{Recall that a homotopical category is a category with a class of weak equivalences satisfying the ``two out of six'' property and containing all identity maps.} which presents the genuine $C_2$-stable homotopy category. This can be refined to a model category as in \cite[Proposition~B.63]{HHR} which we fix for concreteness, however we will not need any specifics involving (co)fibrations and therefore will not discuss the model structure further.
\end{proposition}
Let $\Sp^{BC_2}$ denote the category of orthogonal spectra with $C_2$-action (see \cite{Schwede2016LECTURESOE}). Restricting to the subcategory of trivial $C_2$-representations in $\II_{C_2}$ induces an equivalence of 1-categories \cite[Proposition~A.19]{HHR},
\[
(\;\;)^e \colon \Sp^{C_2} \to \Sp^{BC_2}.
\]

\begin{definition}\label{underlying}
Let $X$ be an orthogonal $C_2$-spectrum. The image, $X^e$, under the map $(\;\;)^e \colon \Sp^{C_2} \to \Sp^{BC_2}$ is called the \textit{underlying spectrum} of $X$.
\end{definition}
Taking homotopy groups of the underlying spectrum recovers the trivial subgroup homotopy groups of $X$, $\pi^e_*(X) \cong \pi_*(X^e)$, along with the induced $C_2$-action. More care is required when discussing the $C_2$-homotopy groups $\pi^{C_2}_*(X)$. The \textit{categorical fixed points} functor 
\[
F^{C_2} \colon \Sp^{C_2} \to \Sp,
\]
is given by 
\[
(F^{C_2}X)(n) = (X(\RR^n))^{C_2}.
\]
However, categorical fixed points does not preserve weak equivalences, as it can fail to capture the $C_2$-homotopy groups of $X$. This can be fixed by taking an appropriate fibrant replacement. 

\begin{definition}\cite[p.~71]{Schwede2016LECTURESOE}
Let $X$ be an orthogonal $C_2$-spectrum. The \textit{$C_2$-fixed points} of $X$ is the orthogonal spectrum $X^{C_2} \in \Sp$ given by $(X^{C_2})(n) = (\Omega^{n \sigma}X(n \rho))^{C_2}$.
\end{definition}

Then we have $\pi^{C_2}_*(X) \cong \pi_*(X^{C_2})$ \cite[Proposition~7.2]{Schwede2016LECTURESOE}. Similarly, we could have modelled $X^e$ by
\[
(X^e)(n) \cong \Omega^{n \sigma}X(n \rho).
\]
Levelwise inclusion of fixed points defines a map, 
\[
X^{C_2} \to X^e,
\]
which upon taking homotopy groups induces the restriction map,
\[
\text{res} \colon \pi^{C_2}_*(X) \to \pi^{e}_*(X),
\]
in the homotopy Mackey functor $\underline{\pi}_*X$. 

\begin{remark}
    The categorical fixed points fit into a Quillen adjunction
    \[
    T_{C_2} \colon \Sp \rightleftarrows \Sp^{C_2} \colon F^{C_2},
    \]
    where the left adjoint comes from giving a spectrum the trivial $C_2$-action and then passing through the equivalence of 1-categories $\Sp^{BC_2} \simeq \Sp^{C_2}$. Deriving this adjunction allows us to identify the cohomology theory represented by $X^{C_2}$ in terms of $X$. 
    \begin{proposition}\label{fixed.cohomology}
     Let $E^*$ be the cohomology theory represented by $X^{C_2}$, and let $Y \in \T$. Then viewing $Y \in \T^{C_2}$ by giving it the trivial $C_2$-action, from the above remark we have
    \[
     E^n(Y) \cong [\Sigma^{\infty}_{C_2} Y, \Sigma^n X]_{C_2}.
    \]
    \end{proposition}

\end{remark}
Now, we consider the multiplicative aspects of our model. Commutative ring structures in the genuine equivariant setting is a rich subject and has been explored by Blumberg and Hill \cite{blumberghill2015operadic}, as well as by others. We will be interested in the complete $N_\infty$-algebras, also called $C_2$-$E_\infty$-ring spectra. The category of $C_2$-$E_\infty$-ring spectra can be modeled using orthogonal $C_2$-spectra as follows. Day convolution equips $\Sp^{C_2}$ with the structure of a closed symmetric monoidal category with monoidal unit given by $S_{C_2}$. As discussed in \cite[Section~2.1]{HHR} and \cite{blumberghill2015operadic}, the category $\CAlg(\Sp^{C_2})$ of commutative monoid objects in $\Sp^{C_2}$ models $C_2$-$E_\infty$-ring spectra. Further, by \cite[Proposition~22.1]{mandell_may_schwede_shipley_2001} it follows that $\CAlg(\Sp^{C_2})$ is equivalent to the category of lax symmetric monoidal $\T^{C_2}$-enriched functors $\II_{C_2} \to \T_{C_2}$. This is the setting in which we will address multiplicative questions. 

\begin{definition}
A \textit{commutative orthogonal $C_2$-ring spectrum} is a lax symmetric monoidal $\T^{C_2}$-enriched functor $\II_{C_2} \to \T_{C_2}$. A \textit{($C_2$-$E_{\infty}$-) map of commutative orthogonal $C_2$-ring spectra} is a symmetric monoidal $C_2$-equivariant enriched natural transformation.
\end{definition}

As is discussed in \cite[p.~77]{Schwede2016LECTURESOE}, the categorical and genuine fixed points are lax symmetric monoidal functors. Thus, for a commutative orthogonal $C_2$-ring spectrum $X$, the spectra $F^{C_2}X$ and $X^{C_2}$ inherit the structure of commutative monoids in $\Sp$. 
\begin{proposition}\label{categorical.genuine}
    There is a symmetric monoidal natural transformation,
    $$
    F^{C_2} \to (\;\;)^{C_2},
    $$
    from the categorical fixed points to the $C_2$-fixed points. 
\end{proposition}

Such a natural transformation can be defined levelwise by letting
\[
\iota \colon \RR^n \to n\rho
\]
be the inclusion of the trivial submodule inside of $n \rho$,
taking the composition
\[
S^{n\sigma} \to \II_{C_2}(\RR^n, n\rho) \to \T_{C_2}(X(\RR^n), X(n\rho)),
\]
and then passing under adjunction to
\[
X(\RR^n) \to \Omega^{n\sigma}X(n\rho),
\]
before finally taking fixed points.

\section{The Real spin bordism spectrum}\label{sec.realspin}

In this section, we construct $\MSpin^c_{\RR}$ as a commutative monoid in orthogonal $C_2$-spectra. This mostly consists of equipping the constructions of Joachim in \cite{Joachim} with Real structures and verifying $C_2$-equivariance wherever it is relevant. The extra data comes from the Real structures on the Clifford algebras $\CCl_{p,q}$ mentioned in Section \ref{Introduction}, combined with the constructions from Section \ref{sec.Real}. 

\vspace{3mm}

%\subsection{Constructing $\MSpin^c_{\RR}$}

Given an inner product space $V$, the Clifford algebra $\Cl(V)$ is characterized by the following universal property. For any linear map $f \colon V \to A$ to an $\RR$-algebra $A$ such that\footnote{see Remark \ref{clifford.signs}} $f(v)^2 = \langle v, v \rangle$, there is a unique extension to an algebra homomorphism $\Tilde{f}$ in
$$
\begin{tikzcd}
V \arrow[d, hook] \arrow[r, "f"] & A \\
\Cl(V). \arrow[ur, dashed, "\Tilde{f}"'] &
\end{tikzcd}
$$

Let $\Cl_n = \Cl(\RR^n)$, where $\RR^n$ is given the standard inner product. The Clifford algebra is a $\Z2$-graded algebra by specifying that $V \subset \Cl(V)$ lies in the odd subspace. Accordingly, the symbol $\otimes$ refers to the graded tensor product. In particular, the symmetry isomorphism  contains a sign,
$$
V \otimes W \xrightarrow{\sim} W \otimes V, \;\;\; v \otimes w \mapsto (-1)^{|v||w|}w \otimes v,
$$
for $v,w$ purely even or odd. With this symmetric monoidal structure, there is an isomorphism $\Cl(V \oplus W) \cong \Cl(V) \otimes \Cl(W)$. For $V \in \II_{C_2}$, the Clifford algebra $\Cl(V)$ inherits an induced $C_2$-action by the universal property above, where $C_2$ acts by algebra maps. We extend this action to a Real structure on the complex Clifford algebra $\CCl(V) := \Cl(V) \otimes \CC$ by complex conjugation on the $\CC$-factor and Proposition \ref{real.tensor}. Note that this $C_2$-action on $\CCl(V)$ preserves the $\Z2$-grading (i.e. it is even), since it comes from an action on $V$. Lastly, fix the $*$-structure on $\CCl(V)$ to be the one generated by $v^* = v$, for $v \in V$.  Let $\CCl_n = \CCl(\RR^n)$, where $\RR^n$ is the trivial $C_2$-representation, and $\CCl_{p,q} = \CCl(\RR^{p,q})$, where $\RR^{p,q} = \RR^p \oplus \RR^q \sigma$. 

\begin{remark}[Clifford sign conventions]\label{clifford.signs}
Sign conventions for Clifford algebras differ significantly throughout the literature. The conventions we use are chosen to be easily comparable to the ones in \cite{Joachim} and \cite{KasparovKK}, whose results we use in this paper. 
Forgetting the Real structure, the underlying complex $C^*$-algebra $\CCl(V)$ agrees with the $C^*$-algebra denoted $\CC\! l_{V}$ in \cite{Joachim}. The Real $C^*$-algebra $\CCl_{p,q}$ is isomorphic to the Real $C^*$-algebra denoted $C_{p,q}$ in \cite{KasparovKK}, as follows. Fix generators $v_1,...,v_p,w_1,...,w_q$ of $\CCl_{p,q}$ and $\varepsilon_1,...,\varepsilon_p,e_1,...,e_q$ of $C_{p,q}$, with 
$$
\begin{aligned}
&v_i^2 = 1, \; &w_j^2 = 1, \;\;\;\;&\varepsilon_i^2 = 1, \;&e_j^2 = -1, & \\
&\overline{v_i} = v_i, &\overline{w_j} = -w_j, \;\;\;& \overline{\varepsilon_i} = \varepsilon_i, &\overline{e_j} = e_j, & \\
&v_i^* = v_i, &w_j^* = w_j, \;\;\;&\varepsilon_i^* = \varepsilon_i, &e_j^* = -e_j. &
\end{aligned}
$$
Then the homomorphism $\CCl_{p,q} \to C_{p,q}$ defined by $v_i \mapsto \varepsilon_i$ and $w_j \mapsto ie_j$ is an isomorphism of Real graded $C^*$-algebras. 
\end{remark}

Let $L^2(V)$ be the $L^2$-completion of the pre-Hilbert space of continuous functions $V \to \CC$ with compact support. Remark \ref{real.hom} defines a Real structure on $L^2(V)$, and Proposition \ref{real.tensor} defines a Real structure on $\CCl(V)\otimes L^2(V)$. Here, we take $L^2(V)$ to be $\Z2$-graded by even and odd functions\footnote{Here, \textit{even} means $f(-x) = f(x)$, and \textit{odd} means $f(-x) = -f(x)$.}. Then the induced Real structure on $\CCl(V)\otimes L^2(V)$ is also even. Let $\BB(\CCl(V) \otimes L^2(V))$ be the Real space of bounded operators on $\CCl(V) \otimes L^2(V)$ equipped with the strong $*$-topology and the Real structure from Remark \ref{real.hom}. 

\begin{definition}
Let $\U_V \subset \BB(\CCl(V) \otimes L^2(V))$ denote the graded unitary group of $\CCl(V) \otimes L^2(V)$, that is, the group consisting of unitary operators on $\CCl(V) \otimes L^2(V)$ that are either grading preserving or grading reversing, and let $\Pin^c(V) \subset \CCl(V)$ be the group generated under multiplication by unit vectors $S(V) \subset V$ and unit complex numbers $\U(1) \subset \CC$.
\end{definition}

Then the $C_2$-actions on $\BB(\CCl(V) \otimes L^2(V))$ and $\CCl(V)$ restrict to $\U_V$ and $\Pin^c(V)$, respectively, making them into Real groups. Since both $C_2$-actions are grading preserving, they restrict to Real structures on the respective even subgroups, $\U_V^{\ev} \subset \U_V$ and $\Spin^c(V) \subset \Pin^c(V)$. By Remark \ref{real.hom}, the orthogonal group $\O(V)$ is a Real group. Twisted conjugation of $\Pin^c(V)$ on $\CCl(V)$ restricts to a homomorphism $\rho_V \colon \Pin^c(V) \to \O(V)$, defined by $\rho_V(g)(v) = (-1)^{|g|}gvg^{-1}$. 

\begin{proposition}\label{prop.rho}
The homomorphism $\rho_V$ is a Real homomorphism.
\end{proposition}
\begin{proof}
Let $g \in \Pin^c(V)$ and $v \in V$. Since the $C_2$-action on $\Pin^c(V)$ is even, we have 
$$
\begin{aligned}
\overline{\rho_V(g)}(v) &= \overline{\rho_V(g)(\overline{v})} = \overline{(-1)^{|g|} g\overline{v}g^{-1}} = (-1)^{|g|}\overline{g}v\overline{g}^{-1} = (-1)^{|\overline{g}|}\overline{g}v\overline{g}^{-1} = \rho_V(\overline{g})(v).
\end{aligned}
$$
\end{proof}

Remark \ref{real.hom} also says that $V$ is a Real $\O(V)$-space, so Proposition \ref{real.props} gives $L^2(V)$ the structure of a Real $\O(V)$-space. By Propositions \ref{real.props} and \ref{prop.rho}, $L^2(V)$ becomes a Real $\Pin^c(V)$-space. Left multiplication also makes $\CCl(V)$ into a Real $\Pin^c(V)$-space, so Proposition \ref{real.tensor} gives us a Real $\Pin^c(V)$-action on $\CCl(V) \otimes L^2(V)$. Since $\Pin^c(V)$ acts on $\CCl(V) \otimes L^2(V)$ by unitary operators, we get an induced Real homomorphism (Proposition \ref{real.props}) $i_V \colon \Pin^c(V) \to \U_V$. Unwinding the definitions yields the formula,
$$i_V(g)(v \otimes f) = (gv) \otimes (f\circ \rho_V(g)^{-1}),$$
as in \cite{Joachim}. This formula makes it clear that $i_V$ is even, so it restricts to a Real homomorphism $$i_V \colon \Spin^c(V) \to \U_V^{\ev}.$$

Thus, we get a Real $\Spin^c(V)$-action  on $\U_V^{\ev}$ by left multiplication. 

\begin{proposition}\label{ESpinc}
If $V \neq 0$, then the $\Spin^c(V)$-space $\U_V^{\ev}$ is a model for the universal $\Spin^c(V)$-principal bundle, $\ESpin^c(V)$. 
\end{proposition}
\begin{proof}
Grading preserving unitary operators can be restricted to the even and odd subspaces to obtain $\U_V^{\ev} \cong \U((\CCl(V) \otimes L^2(V))^{\ev}) \times \U((\CCl(V) \otimes L^2(V))^{\text{odd}})$, which is contractible by Kuiper's theorem \cite{Kuiper}. The $\Spin^c(V)$-action on $\U_V^{\ev}$ is free, since $i_V$ is injective. 
\end{proof}

We will also use $\U^{\ev}_V /\Spin^c(V)$ to model $\BSpin^c(V)$.  By Proposition \ref{real.props}, the associated bundle $\gamma_{\Spin^c(V)} = \U_V^{\ev} \times_{\Spin^c(V)} V$ inherits the structure of a Real space.

\begin{definition}
For $V\in \II_{C_2}$, define the Real space $$\MSpin^c_{\RR}(V) := \text{Thom}(\gamma_{\Spin^c(V)})$$ to be the Thom space of the universal $\Spin^c(V)$-vector bundle, extending the Real structure on $\gamma_{\Spin^c(V)}$ by fixing $\infty$. 
\end{definition}

In order to give $\MSpin^c_{\RR}$ the structure of a $C_2$-$E_\infty$-ring spectrum, we will realize it as a lax symmetric monoidal functor on $\II_{C_2}$. To do this, we first note that given $V, W \in \II_{C_2}$, we get a canonical (even) map $T_{V,W} \colon \U_{V} \times \U_{W} \to \U_{V \oplus W}$, by taking the tensor product of operators, since 
\begin{equation}\label{tensor.clifford}
\begin{aligned}
\CCl(V \oplus W) \otimes L^2(V \oplus W) &\cong (\CCl(V) \otimes \CCl(W)) \otimes (L^2(V) \otimes L^2(W)) \\ &\cong (\CCl(V) \otimes L^2(V)) \otimes (\CCl(W) \otimes L^2(W)).
\end{aligned}
\end{equation}
Similarly, given an isomorphism $\iota \colon V \to W$, we get an isomorphism $\tilde{\iota} \colon \CCl(V)\otimes L^2(V) \to \CCl(W)\otimes L^2(W)$, which  induces a homomorphism $\U_{\iota} \colon \U_{V} \to \U_{W}$, defined by $\U_{\iota}(A) = \tilde{\iota} A \tilde{\iota}^{-1}$. Furthermore, note that these assignments make $\U_{(\;\;)}$ into a functor with respect to isomorphisms. Similarly, $T_{(\;,\;)}$ is a natural transformation. 

\begin{proposition}
Given $V,W \in \II_{C_2}$ of the same dimension, the map  $\U \colon \O(V,W) \to \Hom(\U_{V},\U_{W})$ defined above is $C_2$-equivariant.
\end{proposition}
\begin{proof}
We have
$$
\begin{aligned}
\overline{\U_{\iota}}(A)(v \otimes f) &= \overline{\U_{\iota}(\overline{A})}(v \otimes f) \\ &= \overline{\tilde{\iota}\overline{A}\tilde{\iota}^{-1}(\overline{v\otimes f})} \\
&= \overline{\tilde{\iota}}A\overline{\tilde{\iota}}^{-1}(v\otimes f ) \\
&= \U_{\overline{\iota}}(v \otimes f).
\end{aligned}
$$
\end{proof}

Now, let $V,W \in \II_{C_2}$, and $n = \dim W - \dim V \geq 0$. We define the structure maps 
$$
\II_{C_2}(V,W) \to \T_{C_2}(\MSpin^c_{\RR}(V),\MSpin^c_{\RR}(W)),
$$
using the adjoint form 
$$
\II_{C_2}(V,W) \wedge \MSpin^c_{\RR}(V) \xrightarrow{\sigma_{V,W}} \MSpin^c_{\RR}(W),
$$
by taking the Thomification of the vector bundle map

\[\begin{tikzcd}
(W-V) \times (\U_V^{\ev} \times_{\Spin^c(V)} V) \arrow[r,"\widetilde{\sigma}_{V,W}"] \arrow[d] & \U^{\ev}_{W} \times_{\Spin^c(W)} W \arrow[d] \\
\O(V,W) \times \BSpin^c(V) \arrow[r] & \BSpin^c(W).
\end{tikzcd}\]

Informally, $\widetilde{\sigma}_{V,W}$ is the map
$$
\begin{tikzcd}[column sep = -9mm]
{((\iota,w),[A,v])} \arrow[d,mapsto] & {\in \; (W-V) \times (\U_V^{\ev} \times_{\Spin^c(V)} V)} \arrow[d, "\widetilde{\sigma}_{V,W}"] \\ 
{[T_{W-\iota V,V}(\id,A),w+v]} & {\in \;\; \U^{\ev}_{W} \times_{\Spin^c(W)} W.} 
\end{tikzcd}
$$ 
The issue with this formulation is that $W-\iota V$ is not an object of $\II_{C_2}$. In order to properly define $\widetilde{\sigma}_{V,W}$, we make the $C_2$-equivariant identification
\[
\begin{tikzcd}
W-V \arrow[r,"\sim"] \arrow[d] & \O(\RR^n \oplus V,W) \times_{\O(n)} \RR^n \arrow[d] \\
\O(V,W) \arrow[r,"\sim"] & \O(\RR^n \oplus V,W)/\O(n).
\end{tikzcd}
\]

\begin{definition}
Define
\[
\widetilde{\sigma}_{V,W} \colon (\O(\RR^n \oplus V,W) \times_{\O(n)} \RR^n ) \times (\U_V^{\ev} \times_{\Spin^c(V)} V) \to \U^{\ev}_{W} \times_{\Spin^c(W)} W,
\]
by
\[
\widetilde{\sigma}_{V,W}([\iota,x],[A,v]) = [\U_{\iota}(T_{\RR^n,V}(\id,A)),\iota(x + v)].
\]
\end{definition}

\begin{proposition}
The map $\widetilde{\sigma}_{V,W}$ is $C_2$-equivariant.
\end{proposition}
\begin{proof} We have
\begin{align*}
\sigma_{V,W}(\overline{[\iota,x],[A,v]}) = \sigma_{V,W}([\overline{\iota},\overline{x}],[\overline{A},\overline{v}])
& = [\U_{\overline{\iota}}(T_{\RR^n,V}(\id,\overline{A})),\overline{\iota}(\overline{x}+\overline{v})]
\\
&= [\overline{(\U_{\iota})}(T_{\RR^n,V}(\overline{\id,A})),\overline{\iota(x+v)}]
\\
&=[\overline{\U_{\iota}(T_{\RR^n,V}(\id,A))},\overline{\iota(x+v)}]
\\
&=
\overline{[\U_{\iota}(T_{\RR^n,V}(\id,A)),\iota(x)+\iota(v)]}
\\
&=
\overline{\sigma_{V,W}([\iota,x],[A,v])}.
\end{align*}
\end{proof}

\begin{definition}\label{structure.maps}
Define $\sigma_{V,W} \colon \II_{C_2}(V,W) \wedge \MSpin^c_{\RR}(V) \to \MSpin^c_{\RR}(W)$ as the Thomification of $\widetilde{\sigma}_{V,W}$. 
\end{definition}

To define the unit, note that $\U_0^{\ev} = \U(1)$, so $\MSpin^c_{\RR}(0) = \mathrm{Thom}(\U(1) \times_{\Spin^c(0)} \{0\}) = S^0$. 

\begin{definition}\label{MSpinR.unit}
Let $\eta^{\MSpin^c_{\RR}}_0 \colon S^0 \to \MSpin^c_{\RR}(0)$ be the identity map on $S^0$. 
\end{definition}

\begin{definition}
Let $V,W \in \II_{C_2}$. Define 
$$
\mu^{\MSpin^c_{\RR}}_{V,W} \colon \MSpin^c_{\RR}(V) \wedge \MSpin^c_{\RR}(W) \to \MSpin^c_{\RR}(V \oplus W),  
$$
as the Thomification of the map 
$$
(\U^{\ev}_{V} \times_{\Spin^c(V)} V ) \times (\U^{\ev}_{W} \times_{\Spin^c(W)} W ) \to \U^{\ev}_{V\oplus W} \times_{\Spin^c(V\oplus W)} (V \oplus W) 
$$
defined by $([A,v],[B,w]) \mapsto [T_{V,W}(A,B),v+w]$.
\end{definition}

Note that $\mu^{\MSpin^c_{\RR}}_{V,W}$ is clearly $C_2$-equivariant.

\begin{proposition}\label{realspinbordism}
The data $(\MSpin^c_{\RR},\mu^{\MSpin^c_{\RR}},\eta^{\MSpin^c_{\RR}}_0)$ defines a lax symmetric monoidal enriched functor $\II_{C_2} \to \T_{C_2}$, and hence, a $C_2$-$E_{\infty}$-ring spectrum, $\MSpin^c_{\RR} \in \text{CAlg}(\Sp^{C_2})$. 
\end{proposition}
\begin{proof}

Verifying functoriality of $\MSpin^c_{\RR}$ is equivalent to showing that the following square commutes
\[
\begin{tikzcd}[column sep = 1.7cm]
\II_{C_2}(V_2,V_3) \wedge \II_{C_2}(V_1,V_2) \wedge \MSpin^c_{\RR}(V_1) \arrow[r,"\id \wedge {\sigma}_{V_1,V_2}"] \arrow[d] & \II_{C_2}(V_2,V_3) \wedge \MSpin^c_{\RR}(V_2)\arrow[d, "{\sigma}_{V_2,V_3}"] \\
\II_{C_2}(V_1,V_3) \wedge \MSpin^c_{\RR}(V_1) \arrow[r, "{\sigma}_{V_1,V_3}"] & \MSpin^c_{\RR}(V_3).
\end{tikzcd}
\]
This follows from the fact that both compositions are the Thomification of the same map of bundles

\[\begin{tikzcd}
(V_3-V_2) \times (V_2-V_1) \times (\U_{V_1}^{\ev} \times_{\Spin^c(V_1)} V_1) \arrow[r,"\phi"] \arrow[d] & \U^{\ev}_{V_3} \times_{\Spin^c(V_3)} V_3 \arrow[d] \\
\O(V_2,V_3) \times \O(V_1, V_2) \times \BSpin^c(V_1) \arrow[r] & \BSpin^c(V_3),
\end{tikzcd}\]

where $\phi$ takes
\[((\iota_2, v_3), (\iota_1,v_2),[A,v_1]) \in (V_3-V_2) \times(V_2-V_1) \times (\U_{V_1}^{\ev} \times_{\Spin^c(V_1)} V_1)\] to \[ [T_{V_3 - \iota_2 \iota_1 V_1,V_1}(\id,A),v_3+v_2+v_1] \in \U^{\ev}_{V_3} \times_{\Spin^c(V_3)} V_3.\]

% Naturality:
% $$
% \begin{tikzcd}[column sep={-0.5cm}]
%     \II_{C_2}(V_1,W_1) \wedge \II_{C_2}(V_2,W_2) \arrow[r] \arrow[d] & \T^{C_2}(\MSpin^c_{\RR}(V_1) \wedge \MSpin^c_{\RR}(V_2),\MSpin^c_{\RR}(W_1) \wedge \MSpin^c_{\RR}(W_2)) \arrow[d] \\
%     \T^{C_2}(\MSpin^c_{\RR}(V_1 \oplus V_2),\MSpin^c_{\RR}(W_1 \oplus W_2)) \arrow[r] & \T^{C_2}(\MSpin^c_{\RR}(V_1) \wedge \MSpin^c_{\RR}(V_2),\MSpin^c_{\RR}(W_1 \oplus W_2))
% \end{tikzcd}
% $$
Naturality of $\mu^{\MSpin^c_{\RR}}$ is equivalent to commutativity of 
$$
\begin{tikzcd}[column sep={0.7cm}]
    (\II_{C_2}(V_1,W_1) \wedge \MSpin^c_{\RR}(V_1)) \wedge (\II_{C_2}(V_2,W_2) \wedge  \MSpin^c_{\RR}(V_2)) \arrow[r] \arrow[d] & \MSpin^c_{\RR}(W_1) \wedge \MSpin^c_{\RR}(W_2) \arrow[dd] \\
    (\II_{C_2}(V_1,W_1) \wedge \II_{C_2}(V_2,W_2)) \wedge (\MSpin^c_{\RR}(V_1) \wedge \MSpin^c_{\RR}(V_2)) \arrow[d] & \\
    \II_{C_2}(V_1 \oplus V_2,W_1 \oplus W_2) \wedge \MSpin^c_{\RR}(V_1 \oplus V_2) \arrow[r] & \MSpin^c_{\RR}(W_1 \oplus W_2).
\end{tikzcd}
$$
Going around the top and then the right is the Thomification of the map
\[
(W_1-V_1) \times (W_2-V_2) \times (\U^{\ev}_{V_1} \times_{\Spin^c(V_1)} V_1) \times (\U^{\ev}_{V_2} \times_{\Spin^c(V_2)} V_2 ) \to \U^{\ev}_{W_1 \oplus W_2} \times_{\Spin^c(W_1 \oplus W_2)} (W_1 \oplus W_2),
\]
given by
\[
((\iota_1, w_1),(\iota_2, w_2),[A, v_1], [B, v_2]) \mapsto [T_{W_1, W_2}(T_{W_1-\iota_1V_1}(\id, A),T_{W_2-\iota_2V_2}(\id, B)), w_1+v_1+w_2+v_2].
\]
While going around the left and then the bottom is the Thomification of the map
\[
(W_1-V_1) \times (W_2-V_2) \times (\U^{\ev}_{V_1} \times_{\Spin^c(V_1)} V_1) \times (\U^{\ev}_{V_2} \times_{\Spin^c(V_2)} V_2 ) \to \U^{\ev}_{W_1 \oplus W_2} \times_{\Spin^c(W_1 \oplus W_2)} (W_1 \oplus W_2),
\]
given by
\[
((\iota_1, w_1),(\iota_2, w_2),[A, v_1], [B, v_2]) \mapsto [T_{W_1\oplus W_2 - \iota_1 \oplus \iota_2 (V_1 \oplus V_2), V_1 \oplus V_2}(\id, T_{V_1, V_2}(A, B)), w_1+w_2+ v_1+v_2].
\]

These two maps are equal by naturality of $T$. The associativity, unitality, and symmetry conditions are all clear directly by the constructions. Thus, $\MSpin^c_{\RR}$ is a lax symmetric monoidal functor. 

\end{proof}

\begin{proposition}\label{underlying.spinc}
The underlying spectrum of $\MSpin^c_{\RR}$ is $\MSpin^c$.
\end{proposition}
\begin{proof}
We compare our construction to the one in Section 6 of \cite{Joachim}. Since the underlying spectrum is defined by restricting to trivial $C_2$-representations we suppose that $V$ is a trivial $C_2$-representation. The following two maps of short exact sequences 
$$
\begin{tikzcd}
    1 \arrow[r] & \Spin^c(V) \arrow[r] \arrow[d,"i_V"] & \Pin^c(V) \arrow[r] \arrow[d,"i_V"] & \Z2 \arrow[r] \arrow[d,"="] & 1\\
    1 \arrow[r] & \U^{\ev}_V \arrow[r] & \U_V \arrow[r] & \Z2 \arrow[r] & 1 
\end{tikzcd}
$$
and
$$
\begin{tikzcd}
    1 \arrow[r] & \U(1) \arrow[r] \arrow[d,"="] & \Pin^c(V) \arrow[r] \arrow[d,"i_V"] & \O(V) \arrow[r] \arrow[d,"j_V"] & 1\\
    1 \arrow[r] & \U(1) \arrow[r] & \U_V \arrow[r] & \PU_V \arrow[r] & 1 
\end{tikzcd}
$$

allow us to identify $\gamma_{\Spin^c(V)} = \U^{\ev}_V \times_{\Spin^c(V)} V \cong \PU_V \times_{\O(V)} V$, which shows that the underlying space $\MSpin^c_{\RR}(V)$ is the same as the one denoted by $\mathbb{M}\!\Spin^c_V$ in \cite{Joachim}. Using this identification, we get a map 
$$
k_V \colon \O(V) \times_{\O(V)} V \xrightarrow{j_V \times_{\O(V)} \id_V } \PU_V \times_{\O(V)} V \xrightarrow{\sim} \U^{\ev}_V \times_{\Spin^c(V)} V.
$$ 

The map $\widetilde{\eta}_V \colon S^V \to \MSpin^c_{\RR}(V)$ obtained by taking the Thom construction of $k_V$ is the $V$-component of the lax monoidal natural transformation $S \to \MSpin^c$ playing the role of the unit of the ring spectrum defined in \cite{Joachim}. In our model, the $V$-component of the ring spectrum unit $\eta^{\MSpin^c_{\RR}}_V \colon S^V \to \MSpin^c_{\RR}(V)$ is obtained by combining the structure maps of Definition \ref{structure.maps} with the functor unit  $\eta^{\MSpin^c_{\RR}} \colon S^0 \to \MSpin^c_{\RR}(0)$ from Definition \ref{MSpinR.unit}. Specifically,  $\eta^{\MSpin^c_{\RR}}_V$ is the Thomification of $\widetilde{\sigma}_{0,V} \colon \O(\RR^n,V) \times_{\O(n)} \RR^n \to \U^{\ev}_V \times_{\Spin^c(V)} V$, defined above by $\widetilde{\sigma}_{0,V}([\iota,x]) = [\id, \iota(x)]$. Fixing a specific isomorphism $\varphi \colon \RR^n \xrightarrow{\sim} V$, the diagram
$$
\begin{tikzcd}
\O(\RR^n,V) \times_{\O(n)} \RR^n \arrow[r,"\widetilde{\sigma}_{0,V}"] \arrow[d,"\widetilde{\varphi}"'] & \U^{\ev}_V \times_{\Spin^c(V)} V \\
\O(V) \times_{\O(V)} V \arrow[r, "j_V \times_{\O(V)} \id_V"'] & \PU_V \times_{\O(V)} V, \arrow[u,"\cong"']
\end{tikzcd}
$$
where $\widetilde{\varphi}([\iota,x]) = [\iota\varphi^{-1},\varphi(x)]$, commutes, since $$(j_V \times_{\O(V)} \id_V)([\iota \varphi^{-1}, \varphi(x)]) = [j_V(\iota \varphi^{-1}),\varphi(x)] = [\id,(\iota \varphi^{-1})\varphi(x)] = [\id,\iota(x)].$$
Thus, $\eta^{\MSpin^c_{\RR}}_V = \widetilde{\eta}_V$. It is also evident that the multiplication maps are the same as those in \cite{Joachim}. Thus, $(\MSpin^c_{\RR})^{e} = \MSpin^c$. 
\end{proof}

\begin{remark}
Another way to organize the data of a commutative orthogonal ring spectrum is via the formalism of  $\mathscr{I}$-FSP's \cite{K-theory/0408}. This is the language that Joachim \cite{Joachim} uses; however, it is not hard to see that our constructions agree with his when the models are compared. 
\end{remark}

\section{The fixed points of $\MSpin^c_{\RR}$} \label{sec.fixed}

Recall that the inclusion $\Spin(n) \hookrightarrow \Spin^c(n)$ factors through an isomorphism with the $C_2$-fixed points of $\Spin^c(n) = \Spin^c(n,0)$. In light of this, one might hope for a similar situation on Thom spectra. 

$$
\begin{tikzcd}[column sep = 0mm, row sep = 2mm]
    \Spin(n) \arrow[rr, hook ] \arrow[ddr,"\cong"'] & & \Spin^c(n) && \MSpin \arrow[rr, "u"] \arrow[ddr, dashed, "\Tilde{u}"'] & & \MSpin^c \\ &&& \rightsquigarrow &&& \\
    & \Spin^c(n)^{C_2} \arrow[ruu, hook] & && & (\MSpin^c_{\RR})^{C_2}. \arrow[ruu] &
\end{tikzcd}
$$

Nevertheless, in general, taking (homotopy) fixed points does not commute with taking Thom spectra, so we cannot expect such a factorization via an equivalence (see Remark \ref{geometric.remark} for a discussion of the geometric fixed points). In this section, we show that $u$ does factor through $\tilde{u} \colon \MSpin \to (\MSpin^c_{\RR})^{C_2}$ as above, but that $\Tilde{u}$ cannot be an equivalence (Corollary \ref{cor.counterspin}). 

\begin{lemma}\label{lem.fixedunitary}
    For $V \in \II_{C_2}$ with $\dim V \geq 1$, the space $(\U^{\ev}_{V})^{C_2}$ is contractible. 
\end{lemma}
\begin{proof}
    Restricting a grading preserving operator to the even and odd subspaces yields,  $$(\U^{\ev}_{V})^{C_2} \cong \U(\HH_V^0)^{C_2} \times \U(\HH_V^1)^{C_2},$$ where $\HH_V^0, \HH_V^1$ are the even and odd parts of $\CCl(V) \otimes L^2(V)$, respectively, and $\U(\HH)^{C_2}$ is the group of Real unitary operators on the Real Hilbert space $\HH$. Since $\HH \cong \HH^{C_2} \otimes \CC$ (Proposition \ref{real.tensor.C}),  $\U(\HH)^{C_2}$ can be identified with the orthogonal group $\O(\HH^{C_2})$ of the real Hilbert space $\HH^{C_2}$. By the real version of Kuiper's theorem \cite{Kuiper}, $\O(\HH^{C_2})$ is contractible, which implies $(\U^{\ev}_{V})^{C_2}$ is contractible. 
\end{proof}

\begin{definition}
    Let $\MSpin(n) := \text{Thom}((\U^{\ev}_{\RR^n})^{C_2} \times_{\Spin(n)} \RR^n)$, where $\Spin(n)$ is identified with $\Spin^c(n)^{C_2}$.
\end{definition}

\begin{proposition}
    The spaces $\MSpin(n)$ form an orthogonal spectrum, with the usual structure maps, equivalent to $\MSpin$.
\end{proposition}
\begin{proof}
By Lemma \ref{lem.fixedunitary}, the space $(\U^{\ev}_{\RR^n})^{C_2}$ is contractible with a free $\Spin(n)$-action, and thus is a model for $\ESpin(n)$. From there, this is a standard Thom spectrum construction. For example, one can follow the same steps taken in Section \ref{sec.realspin} with the obvious modifications.
\end{proof}

\begin{definition}
    Let $\Tilde{u}(n) \colon \MSpin(n) \to \MSpin^c_{\RR}(\RR^n)$ be the Thomification of the composite
$$
(\U^{\ev}_{\RR^n})^{C_2} \times_{\Spin(n)} \RR^n \to \U^{\ev}_{\RR^n} \times_{\Spin(n)} \RR^n \to \U^{\ev}_{\RR^n} \times_{\Spin^c(n)} \RR^n.
$$
induced by the inclusions $(\U^{\ev}_{\RR^n})^{C_2} \hookrightarrow \U^{\ev}_{\RR^n}$ and $\Spin(n) \hookrightarrow \Spin^c(n)$. 
\end{definition}

By construction, the image of $\Tilde{u}(n)$ is contained in the $C_2$-fixed points of $\MSpin^c_{\RR}(\RR^n)$.

\begin{proposition}\label{spin.fixedpoints}
    The maps $\Tilde{u}(n)$ define a map of commutative orthogonal ring spectra $$\Tilde{u} \colon \MSpin \to (\MSpin^c_{\RR})^{C_2}.$$
\end{proposition}
\begin{proof}
    The maps $\tilde{u}(n)$ define a symmetric monoidal natural transformation between commutative monoids in $\Sp$ from $\MSpin$ to the categorical fixed points $F^{C_2}\MSpin^c_{\RR}$. By Proposition \ref{categorical.genuine}, this induces a symmetric monoidal natural transformation to $(\MSpin^c_{\RR})^{C_2}$.
\end{proof}

We now show that $\tilde{u}$ is not an equivalence by proving Theorem \ref{intro.counterspin}.  

\subsection{Proof of Theorem \ref{intro.counterspin}}

    Suppose $E$ is a genuine $C_2$-spectrum satisfying conditions (1)-(4) of Theorem \ref{intro.counterspin}. Under the isomorphisms in (1) and (2), the homotopy Mackey functor $\underline{\pi}_*E$ consists of the data
    $$
    \begin{tikzcd}
        \MSpin_*  \arrow[r, bend left = 15, "\text{res}"] & \MSpin^c_* \arrow[l, bend left = 15, "\text{tr}"] \arrow[loop right, "\overline{(\:  \:)}"]
    \end{tikzcd},
    $$
    where $\overline{(\:\:)}$ is the $C_2$-action in (3). Given any class $\alpha \in \MSpin^c_*$, the Mackey functor structure (Definition \ref{Mackey}) implies that the equation,
    \[
    \alpha + \overline{\alpha} = \text{res}(\text{tr}(\alpha)),
    \]
    holds. Then the commutative diagram in (4) implies that for any $\alpha \in \MSpin^c_*$, the oriented bordism class $u^c_*(\alpha + \overline{\alpha})$ must have a spin manifold as a representative. Thus, it is sufficient to show there is some $\alpha \in \MSpin^c_*$ for which $u^c_*(\alpha + \overline{\alpha})$ has no spin representative. Recall the $\hat{A}$-genus (see e.g. \cite{LM}), 
    $$
    \begin{tikzcd}
    {\MSpin_*} \arrow[r, "\hat{A}"] \arrow[d,"u_*"']  & \ZZ \arrow[d,hook] \\
    {\MSO_*} \arrow[r, "\hat{A}"] & \QQ,
    \end{tikzcd}
    $$
    and consider the class
    \[
    \gamma = [\CP^2] \in \MSpin^c_4.
    \]
    We know that $\hat{A}(u^c_*(\gamma)) = -1/8$ (\cite{LM}). By the discussion above, the class $\gamma + \overline{\gamma}$ must be in the image of $\MSpin_4$. Then $\hat{A}(u^c_*(\gamma + \overline{\gamma})) =: n\in \ZZ$, so 
    \[
    \hat{A}(u^c_*(\overline{\gamma})) =  \frac{1}{8} + n.
    \]
    Now consider the class
    \[
    \alpha = \gamma^2 \in \MSpin^c_8.
    \]
    By multiplicativity of both the $C_2$-action and $\hat{A}$, we see that
    \[
    \hat{A}(u^c_*(\alpha + \overline{\alpha})) = \frac{1}{64} + \frac{1}{64} + \frac{2n}{8} + n^2 = \frac{8n+1}{32} + n^2,
    \]
    which is not an integer, since $32$ never divides $8n+1$. Thus, $\alpha + \overline{\alpha}$ is not in the image of $\text{res} \colon \MSpin_* \to \MSpin^c_*$, which is a contradiction. This completes the proof of Theorem \ref{intro.counterspin}.

\begin{remark}
As can be seen from the proof, condition (4) in Theorem \ref{intro.counterspin} can be loosened to just requiring that the map $\text{res} \colon \MSpin_* \to \MSpin^c_*$ preserves the $\hat{A}$-genus. 
\end{remark}

\begin{corollary}\label{cor.counterspin}
If there is an equivalence $\MSpin \xrightarrow{\sim} (\MSpin^c_{\RR})^{C_2}$, then the composite, $$\MSpin \xrightarrow{\sim} (\MSpin^c_{\RR})^{C_2} \to \MSpin^c,$$ does not take a spin manifold to its underlying spin${}^c$ manifold on homotopy groups. It follows that $\tilde{u}$ of Proposition \ref{spin.fixedpoints} is not an equivalence.
\end{corollary}

\begin{remark}\label{geometric.remark}
Recall that, in contrast with (homotopy) fixed points, taking geometric fixed points does commute with taking Thom spectra. Indeed, the geometric fixed points of Real bordism is unoriented bordism, $\Phi^{C_2}\MU_{\RR} \simeq \MO$. In analogy, one might expect an equivalence between $\Phi^{C_2}\MSpin^c_{\RR}$ and $\MSpin$, but Proposition \ref{Real.orientation} (below) implies that such an equivalence would induce a ring map $\MO \to \MSpin$ (by lax monoidality of $\Phi^{C_2}$), which doesn't exist. So $\MSpin \not\simeq \Phi^{C_2}(\MSpin^c_{\RR})$, and similarly, $\MSpin \not\simeq \Phi^{C_2}((\MSpin^c_{\RR})^h)$.
\end{remark}

\section{A Real orientation of Real spin bordism}\label{sec.MR}

In this section, we show that the complex orientation of $\MSpin^c$ can be refined to a Real orientation of $\MSpin^c_{\RR}$.

\begin{proposition}\label{Real.orientation}
There is a map of $C_2$-$E_{\infty}$-ring spectra $\MU_{\RR} \to \MSpin^c_{\RR}$, where $\MU_{\RR}$ denotes the Real bordism spectrum. Thus, $\MSpin^c_{\RR}$ is Real oriented.
\end{proposition}

The next two subsections are dedicated to proving Proposition \ref{Real.orientation}.

\subsection{Real spectra}
We begin by briefly recalling another model for $C_2$-spectra which will be convenient for constructing the Real orientation. This model is not needed anywhere else in this paper. For all details, we refer the reader to the appendix of Hill, Hopkins, and Ravenel \cite[Section~B.12]{HHR} or Schwede \cite[Chapter~2]{Schwede2016LECTURESOE}. Since a $C_2$-spectrum is determined by its values on multiples of a regular representation, we can instead \textit{define} $C_2$-spectra by indexing over this smaller sequence of representations. This idea is made precise as follows. 
Define a $\T^{C_2}$-enriched category $\II_{\RR}$ with objects $\RR^n =: n$, and with the $C_2$-mapping space $\II_{\RR}(n, m)$ defined as the Thom space of the difference bundle associated to the complexifications,
\[
\II_{\RR}(n, m) = \text{Thom}(\U(\CC^n, \CC^m); \CC^m- \CC^n),
\]
with $C_2$-action given by complex conjugation.
\begin{definition}
    The category of \textit{Real spectra}, $\Sp_{\RR}$, is the category of $\T^{C_2}$-enriched functors $$\II_{\RR} \to \T_{C_2}.$$
\end{definition}
Given $X \in \Sp_{\RR}$, $k \in \ZZ$, and a subgroup $H \subset C_2$, define the homotopy group
\[
\pi_k^H(X) = \underset{n}{\colim} \: \pi_{k + n\rho}^H(X(n)).
\]
Define a weak equivalence between Real spectra to be a map which induces an isomorphism on homotopy groups. This homotopical category structure can be refined to a model category \cite{hill_hopkins_ravenel_2021}. Define a $\T^{C_2}$-enriched functor $i \colon \II_{\RR} \to \II_{C_2}$ by
\[
n \mapsto n\rho = \RR^n \otimes \rho,
\]
and the natural inclusion on mapping $C_2$-spaces.
\begin{proposition}\cite[Proposition~B.226]{HHR}\label{quillen.real}
There is a Quillen equivalence,
\[
i_! \colon \Sp_{\RR} \rightleftarrows \Sp^{C_2} \colon i^*,
\]
given by left Kan extension along $i$ and precomposition by $i$ respectively.   
\end{proposition}
 
The left adjoint is also strongly symmetric monoidal, so we also get a Quillen equivalence between categories of commutative monoids \cite[Proposition~B.248]{HHR},
\[
i_! \colon \CAlg(\Sp_{\RR}) \rightleftarrows \CAlg(\Sp^{C_2}) \colon i^*.
\]
Because of this Quillen equivalence, $C_2$-$E_\infty$-ring maps can be modelled by lax symmetric monoidal natural transformations between lax symmetric monoidal functors out of $\II_{\RR}$.

\subsubsection{$\MU_{\RR}$ and Real orientations}
As is discussed in \cite[Section~B.12]{HHR} and \cite[Chapter~2]{Schwede2016LECTURESOE}, the most natural way to construct $\MU_{\RR}$ as a $C_2$-$E_{\infty}$-ring spectrum is by constructing it as a commutative monoid object in Real spectra and then applying the Quillen equivalence above. This can be done by equipping the spaces $\MU(n)$ with the complex conjugation actions, which defines an object $\mathcal{MU}_{\RR} \in \CAlg(\Sp_{\RR})$. As in \cite[Definition~B.251]{HHR}, $\MU_{\RR}$ is then defined as a commutative orthogonal $C_2$-ring spectrum by applying $i_!$ to a cofibrant replacement of $\mathcal{MU}_{\RR}$ in $\CAlg(\Sp_{\RR})$.

\begin{remark}
For $X \in \CAlg(\Sp^{C_2})$, Real $E_{\infty}$-orientations of $X$ are represented by commutative ring maps $\MU_{\RR} \to X$. By the Quillen equivalence in Proposition \ref{quillen.real}, it is equivalent to consider lax symmetric monoidal natural transformations $\mathcal{MU}_{\RR} \to i^*(X)$ in $\Sp_{\RR}$, where $i^*(X)$ is the Real spectrum obtained from $X$ by restricting to regular representations.
\end{remark}

\subsection{$\MSpin^c_{\RR}$ is Real oriented}

Recall that the homomorphism $\iota \times \det \colon \U(n) \to \text{SO}(2n) \times \U(1)$ lifts to a homomorphism $\varphi(n) \colon \U(n) \to \Spin^c(2n)$. Since $\iota$ is injective, $\varphi(n)$ is also injective. Furthermore, observe that $\varphi(n)$ is actually a Real homomorphism of Real groups $\varphi(n) \colon \U(n) \to \Spin^c(n\rho)$. By Proposition \ref{ESpinc}, the $\U(n)$-space $\U^{\text{even}}_{n\rho}$ is a model for $\text{EU}(n)$ as a Real $\U(n)$-space. The associated bundle $\gamma_{\U(n)} = \U^{\ev}_{n\rho} \times_{\U(n)} \CC^n$ then also inherits a Real structure and the identity map on $\U^{\ev}_{n\rho} \times \CC^n = \U^{\ev}_{n\rho} \times (n\rho)$ descends to a Real map of Real vector bundles $\gamma_{\varphi(n)} \colon \gamma_{\U(n)} \to \gamma_{\Spin^c(n\rho)}$.  

\begin{definition}
Define $\mathcal{MU}_{\RR}'(n)$ to be the Thom space of $\gamma_{\U(n)}$, with its induced Real structure, and let $M\varphi(n) \colon \mathcal{MU}_{\RR}'(n) \to \MSpin^c_{\RR}(n\rho)$ be the Thomification of $\gamma_{\varphi(n)}$. 
\end{definition}

\begin{proposition}
    The Real spaces $\mathcal{MU}'_{\RR}(n)$ (with the usual structure maps) define a lax symmetric monoidal enriched functor $\mathcal{MU}'_{\RR} \colon \II_{\RR} \to \T_{C_2}$, and thus, $\mathcal{MU}'_{\RR} \in \CAlg(\Sp_{\RR})$.
\end{proposition}
\begin{proof}
    This is a standard Thom spectrum construction, which is analogous to the construction in Section \ref{sec.realspin}.
\end{proof}

\begin{proposition}\label{ringmap.MUMSpinc}
        The maps $M\varphi(n)$ define a map of commutative monoids 
        \[
        M\varphi \colon \mathcal{MU}'_{\RR} \to i^*\MSpin^c_{\RR}\]
        in $\Sp_{\RR}$.
\end{proposition}
\begin{proof}
This follows from another standard Thom spectrum argument, using the fact that the squares,
\[
\begin{tikzcd}[row sep = 2mm, column sep = 3mm]
\U(n) \ar[rr] \ar[dd] && \Spin^c(n\rho) \ar[dd] & & \U(n)\times \U(m) \ar[rr] \ar[dd] && \Spin^c(n\rho) \times \Spin^c(m\rho) \ar[dd]  \\ &&& {\text{and}} &&& \\
\U(n+1)  \ar[rr] && \Spin^c((n+1)\rho) & & \U(n+m)  \ar[rr] && \Spin^c((n+m)\rho),
\end{tikzcd}
\]
commute.
\end{proof}

The last step we need in order to show that $\MSpin^c_{\RR}$ is Real oriented is to show that $\mathcal{MU}'_{\RR}$ is equivalent to $\mathcal{MU}_{\RR}$. By construction, we know that the space $\mathcal{MU}_{\RR}'(n)$ is equivalent to $\MU(n)$. However, it is not obvious that the $C_2$-action we constructed on $\mathcal{MU}_{\RR}'(n)$ is equivalent to the complex conjugation action on $\MU(n)$. The rest of this section is devoted to showing that our construction does model $\MU$ with complex conjugation. 

\begin{lemma}\label{fixedpoints.MU}
    If $X$ and $Y$ are (right and left, respectively) Real $\U(n)$-spaces, with $\U(n)$ acting freely on $X$, then the $C_2$-fixed points of $X \times_{\U(n)} Y$ is given by the space
    \[
    (X \times_{\U(n)} Y)^{C_2} \cong X^{C_2} \times_{\O(n)} Y^{C_2}.
    \]
\end{lemma}
\begin{proof}
    Define the map
    \[
    \phi \colon X^{C_2} \times_{\O(n)} Y^{C_2} \to (X \times_{\U(n)} Y)^{C_2},
    \]
    to be the one induced by the inclusions $X^{C_2} \hookrightarrow X$ and $Y^{C_2} \hookrightarrow Y$. That is, we have that  $\phi([x,y]_{\O(n)}) = [x,y]_{\U(n)}$. Here the notation $[x,y]_G$ is referring to the class represented by $(x,y)$ in $X \times_G Y$.
    First, we show injectivity. Suppose $\phi([x_0,y_0]_{\O(n)}) = \phi([x_1, y_1]_{\O(n)})$. Then there exists $U \in \U(n)$, such that
    \[
    (x_0, y_0) = (x_1 U^* , U y_1 ).
    \]
    But since $x_0, x_1$ are fixed by $C_2$ we get
    \[
    x_1 U^* = x_0 = \overline{x_0},  = \overline{x_1} \overline{U}^{*} = x_1 \overline{U}^{*}.
    \]
    Thus, since $G$ acts freely on $X$, we must have that $U = \overline{U}$. So $U \in \O(n)$ and
    \[
    [x_0,y_0]_{\O(n)} = [x_1, y_1]_{\O(n)}.
    \]
    Next, we show surjectivity. Let $[x, y]_{\U(n)} \in (X \times_{\U(n)} Y)^{C_2}$. Then there exists $U \in \U(n)$ such that
    \[
    (\overline{x}, \overline{y}) = (xU^{*},Uy).
    \]
    Observe that from $\overline{x} = x U^{*}$,
    we obtain
    \[
    x = \overline{x U^{*}} = \overline{x} \overline{U^{*}} = \overline{x}U^T,
    \]
    by taking conjugates. However, if we instead acted by $U$ on both sides of $\overline{x} = x U^{*}$, we would obtain $\overline{x} U = x.$
    By freeness of the action of $\U(n)$ on $X$, we must have that $U = U^T$, so $U$ is symmetric. Then we may choose an $S \in \U(n)$ such that $S^2 = U$ and $S^T = S$. This can be done by diagonalizing $U$ and choosing a square root for each eigenvalue of $U$ \cite[Theorem~5.12]{Zhang1999}.
    Now consider the element
    \[
    (w, z) = (x S^{*}, Sy).
    \]
    By construction, we have that
    \[
    [w, z]_{U(n)} = [x, y]_{U(n)}.
    \]
    Observe that both $w$ and $z$ are fixed points, since
    \[
    \overline{w} = \overline{x S^{*}} = \overline{x} \overline{S^{*}} = \overline{x} S = x U^{*} S = x S^* = w,
    \]
    and 
    \[
    \overline{z} = \overline{Sy} =  \overline{S}\overline{y} = S^{*}\overline{y}  = S^{*} U y = Sy = z.
    \]
    It follows that $[x,y]_{U(n)}$ is in the image of $\phi$.
    Continuity of the inverse is then immediate from the fact that open sets in both the source and the target can be represented by the set of orbits of an open set in $X^{C_2} \times Y ^{C_2}$.
\end{proof}

\begin{remark}
Lemma \ref{fixedpoints.MU} depends on the ability to choose square roots in the group $\U(n)$ that interact appropriately with the Real structures, and is thus not true for an arbitrary Real group $G$. In particular, the proof does not apply to $G = \Spin^c(n)$.
\end{remark}

\begin{proposition}
    There is a weak equivalence in $\CAlg(\Sp_{\RR})$,
    \[
        \mathcal{MU}'_{\RR} \simeq \mathcal{MU}_{\RR}.
    \]
\end{proposition}
\begin{proof}
    By construction, we know that we have a nonequivariant homotopy equivalence
    \[
    \mathcal{MU}'_{\RR}(n) \simeq \MU(n) = \mathcal{MU}_{\RR}(n),
    \]
     but we need to verify that the $C_2$-action that we have defined on $\mathcal{MU}'_{\RR}(n)$ is equivalent to complex conjugation.
     Recall that complex conjugation on the universal complex vector bundle $\gamma_n = \text{EU}(n) \times_{\U(n)} \CC^n $ over $\BU(n)$ defines a universal $n$-dimensional Real vector bundle \cite{Edelson}. Then the Real bundle $\U^{\ev}_{n\rho} \times_{\U(n)} \CC^n \to \U^{\ev}_{n{\rho}}/\U(n)$, induces a map of Real vector bundles which fits into a pullback diagram of $C_2$-spaces
     \begin{equation}\label{pullback.complexvb}
\begin{tikzcd}
\U^{\ev}_{n\rho} \times_{\U(n)} \CC^n \arrow[r] \arrow[d] & \text{EU}(n) \times_{\U(n)} \CC^n \arrow[d] \\
\U^{\ev}_{n{\rho}}/\U(n) \arrow[r] & \BU(n).
\end{tikzcd}
\end{equation}
We define
\[
\theta \colon \mathcal{MU}'_{\RR} \to \mathcal{MU}_{\RR},
\]
by letting
\[
\theta_n \colon \mathcal{MU}'_{\RR}(n) \to \mathcal{MU}_{\RR}(n),
\]
be the $C_2$-equivariant map given by Thomification of the map of total spaces, 
\[
\beta_n \colon \U^{\ev}_{n\rho} \times_{\U(n)} \CC^n \to \text{EU}(n) \times_{\U(n)} \CC^n,
\]
in our map of Real vector bundles.
We will verify that $\theta_n$ is an equivalence of $C_2$-spaces. To do this, we first observe that $\theta_n$ is an equivalence on underlying spaces. Since $\beta_n$ fits into the pullback diagram of complex vector bundles \eqref{pullback.complexvb},
the map on base spaces must be a homotopy equivalence, since this pullback square classifies that $\U^{\ev}_{n\rho} \times_{\U(n)} \CC^n$ is a model for the universal $n$-dimensional complex vector bundle. Thus, the map on Thom spaces must also be an equivalence. On fixed points, we use Lemma \ref{fixedpoints.MU} (once on the total space and once on the base space) to observe that the real vector bundle
\[
(\U^{\ev}_{n\rho} \times_{\U(n)} \CC^n)^{C_2} \to (\U^{\ev}_{n{\rho}}/\U(n))^{C_2},
\]
can be identified with
\[
(\U^{\ev}_{n\rho})^{C_2} \times_{\O(n)} \RR^n \to (\U^{\ev}_{n{\rho}})^{C_2}/\O(n).
\]
Then using the fact that taking fixed points preserves pullback squares, we see that $\beta_n^{C_2}$ fits into a pullback square giving a map of real vector bundles 
\[\begin{tikzcd}
(\U^{\ev}_{n\rho})^{C_2} \times_{\O(n)} \RR^n \arrow[r] \arrow[d] & \text{EO}(n) \times_{\O(n)} \RR^n \arrow[d] \\
(\U^{\ev}_{n{\rho}})^{C_2}/\O(n) \arrow[r] & \text{BO}(n).
\end{tikzcd}\]
By Lemma \ref{lem.fixedunitary}, $(\U^{\ev}_{n\rho})^{C_2}$ is a contractible space with a free $\O(n)$-action, so this pullback square classifies the universal real vector bundle on the model of $\text{BO}(n)$ given by $(\U^{\ev}_{n{\rho}})^{C_2}/\O(n)$. Thus, the map on Thom spaces,  $\theta_n^{C_2}$, is an equivalence. We conclude that $\theta_n$ is an equivalence of $C_2$-spaces for every $n$. Monoidality and naturality of this construction are clear, so the maps $\theta_n$ give an equivalence
\[
\theta \colon \mathcal{MU}'_{\RR} \simeq \mathcal{MU}_{\RR}
\]
in $\CAlg(\Sp_{\RR})$.
\end{proof}

This completes the proof of Proposition \ref{Real.orientation}.

\section{A Real spin orientation of Real $\mathrm{K}$-theory}\label{sec.KR}

In this section, we refine the Atiyah--Bott--Shapiro spin${}^c$ orientation of $\KU$ (\cite{ABS}, \cite{Joachim}) to a $C_2$-$E_{\infty}$-map from $\MSpin^c_{\RR}$ to Atiyah's Real K-theory. By Proposition \ref{spin.fixedpoints}, this recovers the Atiyah--Bott--Shapiro spin orientation of KO as an $E_{\infty}$-map as well. We start by adapting the construction of KU in \cite{Joachim} (for trivial $G$) to a model for Real K-theory as a $C_2$-$E_{\infty}$-ring spectrum, $\KU_{\RR}$. Due to the adaptability of Kasparov's KK-theory to the setting of Real $C^*$-algebras, our modifications consist of equipping all relevant $C^*$-algebras with appropriate Real structures. Following \cite{Joachim}, we implicitly use the formalism of $\mathscr{I}_{C_2}$-FSP's \cite{K-theory/0408} for ease of comparison of our constructions.  

\vspace{3mm}

Let $\s = C_0(\RR)$ be the $\Z2$-graded $C^*$-algebra of continuous functions $\RR \to \CC$ vanishing at infinity, graded by even and odd functions. We equip $\s$ with a Real structure defined by complex conjugation on values. 

\begin{proposition}
The $C^*$-algebra $\s$ is generated by the functions $a,b \in \s$ defined by $$a(t) = \dfrac{1}{1+t^2} \;\;\; \text{ and }\;\;\; b(t) = \dfrac{t}{1+t^2}.$$ 
\end{proposition}

Given a $\Z2$-graded Hilbert space $\HH$, let $\KK(\HH)$ denote the $\Z2$-graded $C^*$-algebra of compact operators on $\HH$. For $V \in \II_{C_2}$, let $\KK_V = \KK(L^2(V))$. Then Remark \ref{real.hom} gives $\KK_V$ a Real structure, and Proposition \ref{real.tensor} gives $\CCl(V)\otimes \KK_V$ a Real structure. If $A,B$ are any Real $\Z2$-graded $C^*$-algebras, then Remark \ref{real.hom} also makes the space of even $*$-homomorphisms $C^*_{\text{gr}}(A,B)$ into a Real space.

\begin{definition}
For $V \in \II_{C_2}$, define the Real space $$\KU_{\RR}(V) := C^*_{\text{gr}}(\s,\CCl(V) \otimes \KK_V).$$ 
\end{definition}

\begin{remark}\label{remark.iso}
For Hilbert spaces $\HH_1,\HH_2$, there is a canonical isomorphism 
\begin{equation}\label{compact.tensor}
\KK(\HH_1)\otimes \KK(\HH_2) \cong \KK(\HH_1 \otimes \HH_2).
\end{equation}
Combining this with the isomorphism $\CCl(V) \xrightarrow{\sim} \KK(\CCl(V))$ given by Clifford multiplication, yields a canonical isomorphism
\begin{equation}\label{eqn.iso}
\KK(\CCl(V)\otimes L^2(V)) \cong \CCl(V) \otimes \KK_V. 
\end{equation}
Each of the isomorphisms above respect the Real structures obtained from the constructions in Section \ref{sec.Real}, so the isomorphism in \eqref{eqn.iso} respects the Real structures as well. Hence, we may identify the Real spaces $\KU_{\RR}(V) = C^*_{\gr}(\s,\KK(\CCl(V)\otimes L^2(V)))$.  
\end{remark}

To construct the multiplication maps $\mu^{\KU_{\RR}}_{V,W}$, notice that there is a coassociative and cocommutative comultiplication on $\s$,
$$
\Delta \colon \s = C_0(\RR) \xrightarrow{\Tilde{\Delta}} C_0(\RR^2) \xrightarrow{m^{-1}} C_0(\RR) \otimes C_0(\RR) = \s \otimes \s,
$$
where $\Tilde{\Delta}(a)(x,y) = \dfrac{1}{1+x^2+y^2}$, $\Tilde{\Delta}(b)(x,y) = \dfrac{x+y}{1+x^2+y^2}$, and $m(f \otimes g)(x,y) = f(x)g(y)$. If we endow $C_0(\RR^2)$ with the Real structure given by complex conjugation on values, then it is clear that $\Delta$ is Real. 

\begin{definition}\label{def.star}
Given Real $\Z2$-graded $C^*$-algebras $A, B$, define the Real map
$$
\star \colon C^*_{\gr}(\s,A) \wedge  C^*_{\gr}(\s,B) \to  C^*_{\gr}(\s,A \otimes B) 
$$
to be the composite of Real maps
$$
C^*_{\gr}(\s,A) \wedge  C^*_{\gr}(\s,B) \xrightarrow{\otimes} C^*_{\gr}(\s \otimes \s,A \otimes B) \xrightarrow{\Delta^*} C^*_{\gr}(\s,A \otimes B). 
$$
\end{definition}

\begin{definition} Let $V, W \in \II_{C_2}$. Define the Real map
\[
\mu^{\KU_{\RR}}_{V,W} \colon \KU_{\RR}(V) \wedge \KU_{\RR}(W) \to \KU_{\RR}(V\oplus W),
\]
via the $\star$ in Definition \ref{def.star} with $A = \KK(\CCl(V)\otimes L^2(V))$ and $B = \KK(\CCl(W)\otimes L^2(W))$, 
where we use the identifications in \eqref{eqn.iso}, \eqref{compact.tensor}, and \eqref{tensor.clifford}.
\end{definition}

Next, we will define the unit, $\eta^{\KU_{\RR}}$, using functional calculus. 

\begin{definition}
For $V \in \II_{C_2}$, define $\fc_V \colon V \to C^*_{\gr}(\s,\CCl(V))$, as follows: for $v \in V$, 
$$
\fc_V(v)(a) = \dfrac{1}{1+|v|^2} \;\;\; \text{ and }\;\;\; \fc_V(v)(b) = \dfrac{v}{1+|v|^2}.
$$ 
\end{definition}

\begin{proposition}
The map $\fc_V$ is $C_2$-equivariant. 
\end{proposition}
\begin{proof}
On generators, $$\overline{\fc_V(v)}(a) = \overline{\fc_V(v)(\overline{a})} = \overline{\fc_V(v)(a)} = \fc_V(v)(a) = \fc_V(\overline{v})(a),$$
and 
$$
\overline{\fc_V(v)}(b) = \overline{\fc_V(v)(\overline{b})} = \overline{\fc_V(v)(b)} = \dfrac{\overline{v}}{1+|v|^2} = \dfrac{\overline{v}}{1+|\overline{v}|^2} = \fc_V(\overline{v})(b). 
$$
\end{proof}

Since $\fc_V(v) \to 0$ as $v \to \infty$, we get an induced $C_2$-equivariant map 
$$
\beta_V \colon S^V \to C^*_{\gr}(\s,\CCl(V))
$$
of Real spaces. Let $P_{e^{-|\cdot|^2}} \in \KK_V$ be the projection on the Gaussian $e^{-|\cdot |^2} \in L^2(V)$. 

\begin{definition}
Define $p_V \in C^*_{\gr}(\s,\KK_V)$ by
$$
p_V(f) = f(0)P_{e^{-|\cdot|^2}},
$$
i.e. the composite
$$
\s \xrightarrow{\text{eval}_0} \CC \xrightarrow{P_{e^{-|\cdot|^2}} } \KK_V.
$$
\end{definition}

Note that $p_V$ is invariant with respect to the Real structure of $C^*_{\gr}(\s,\KK_V)$. Then $\beta_V$ and $p_V$ can be combined to obtain a Real map
$$
\eta^{\KU_{\RR}}_V \colon S^V \to \KU_{\RR}(V),
$$
defined as the composite 
$$
S^V \cong S^V \wedge S^0 \xrightarrow{\beta_V \wedge p_V} C^*_{\gr}(\s,\CCl(V)) \wedge C^*_{\gr}(\s,\KK_V) \xrightarrow{\star} C^*_{\gr}(\s,\CCl(V) \otimes \KK_V) = \KU_{\RR}(V).
$$

\begin{proposition}
The data $(\KU_{\RR}(V),\mu^{\KU_{\RR}},\eta^{\KU_{\RR}})$ define a commutative orthogonal $C_2$-ring spectrum, $\KU_{\RR}$, with underlying spectrum KU.
\end{proposition}
\begin{proof} Nonequivariantly, all of our constructions agree with those of \cite{Joachim}, so the only additional content here is the fact that $\mu^{\KU_{\RR}}_{V,W}$ and $\eta^{\KU_{\RR}}_V$ are $C_2$-equivariant, for $V,W \in \II_{C_2}$. This was shown above as each map was constructed. 
\end{proof}

\begin{proposition}\label{KRisKR}
The $C_2$-spectrum $\KU_{\RR}$ represents Atiyah's Real K-theory. 
\end{proposition}
\begin{proof} Since we know that the underlying spectrum of $\KU_{\RR}$ is $\KU$ it is sufficient to show that the $C_2$-fixed points of $\KU_{\RR}$ is $\KO$. Let $X \in \T \subset \T^{C_2}$, and set $S^{p,q} = S^{\RR^{p,q}}$. By Lemma 5.8 of \cite{Joachim}, there is a homeomorphism of spaces,

$$
\begin{aligned}
\T_{C_2}(S^{p,q} \wedge X, \KU_{\RR}(\RR^{p+n,q})) &=  \T_{C_2}(S^{p,q} \wedge X, C^*_{\gr}(\s, \CCl_{p+n,q} \otimes \KK_{\RR^{p+n,q}})) \\
&\cong C^*_{\gr}(\s, C_0(S^{p,q} \wedge X) \otimes \CCl_{p+n,q} \otimes \KK_{\RR^{p+n,q}}) \\
&\cong C^*_{\gr}(\s, C_0(X) \otimes C_0(S^{p,q}) \otimes \CCl_{p,q} \otimes \CCl_{n} \otimes \KK_{\RR^{p+n,q}}),
\end{aligned}
$$
which is also clearly $C_2$-equivariant. Recall that the category of Real $C^*$-algebras is equivalent to the category of real $C^*$-algebras via fixed points in one direction and complexification in the other \cite{Kasparov80}. Thus, taking fixed points yields
$$
\begin{aligned}
\T^{C_2}(S^{p,q} \wedge X, \KU_{\RR}(\RR^{p+n,q})) &\cong C^*_{\gr}(\s, C_0(X) \otimes C_0(S^{p,q}) \otimes \CCl_{p,q} \otimes \CCl_{n} \otimes \KK_{\RR^{p+n,q}})^{C_2} \\
&\cong C^*_{\gr}(\s^{C_2}, C_0(X)^{C_2} \otimes C_0(S^{p,q})^{C_2} \otimes (\CCl_{p,q})^{C_2} \otimes \Cl_{n} \otimes (\KK_{\RR^{p+n,q}})^{C_2}).
\end{aligned}
$$
Let KKO denote Kasparov's real KK-functor and let KKR denote Kasparov's ``Real" KK-functor \cite{KasparovKK}.
By the real version of the main theorem in \cite{Trout}, we have 
$$
\pi_0 C^*_{\gr}(\s^{C_2}, A \otimes (\KK_{\RR^{p+n,q}})^{C_2}) \cong \text{KKO}(A),
$$

so taking homotopy classes, we get
$$
\begin{aligned}
[S^{p,q} \wedge X, \KU_{\RR}(\RR^{p+n,q})]_{C_2} &\cong \text{KKO}(C_0(X)^{C_2} \otimes C_0(S^{p,q})^{C_2} \otimes (\CCl_{p,q})^{C_2} \otimes \Cl_{n}) \\
&\cong \text{KKR}(C_0(X) \otimes C_0(S^{p,q}) \otimes \CCl_{p,q} \otimes \CCl_{n}), 
\end{aligned}
$$
using the equivalence of real $C^*$-algebras and Real $C^*$-algebras. By Kasparaov's Real Bott periodicity \cite{KasparovKK},
$$
\text{KKR}(A) \cong \text{KKR}(C_0(S^{V}) \otimes \CCl(V) \otimes A ),
$$
we have
$$
\begin{aligned}
\text{KKR}(C_0(X) \otimes C_0(S^{p,q}) \otimes \CCl_{p,q} \otimes \CCl_{n}) &\cong \text{KKR}(C_0(X) \otimes \CCl_n) \\
&\cong \text{KKO}(C_0(X)^{C_2} \otimes \Cl_n). 
\end{aligned}
$$
Now, we use the fact that the constructions in \cite{Joachim}, when applied to real $C^*$-algebras, represent KO-theory, i.e. the spaces $\KO(n) := C^*_{\gr}(\s^{C_2}, \Cl_n \otimes (\KK_{\RR^n})^{C_2})$ form a spectrum (in a way analogous to how we defined $\KU_{\RR}$) representing KO. Since $C_2$ acts trivially on $X$, $C_0(X)^{C_2}$ is the algebra of real-valued continuous functions on $X$, so
$$
\begin{aligned}
\text{KKO}(C_0(X)^{C_2} \otimes \Cl_n) &\cong [X, C^*_{\gr}(\s^{C_2}, \Cl_n \otimes (\KK_{\RR^n})^{C_2})]\\
&\cong [X,\KO(n)]. 
\end{aligned}
$$

By putting the above isomorphisms together and taking a colimit, we find that $$\KU_{\RR}^{*}(X) \cong \KO^*(X),$$ for all $X$ with trivial $C_2$-action. By Proposition \ref{fixed.cohomology}, this implies $(\KU_{\RR})^{C_2} \simeq \KO$, which completes the proof. 
\end{proof}

We proceed to construct the Real spin orientation of $\KU_{\RR}$. In the following definition, we use the identification in Remark \ref{remark.iso} to view $\eta^{\KU_{\RR}}_V$ as a map $S^V \to C^*_{\gr}(\s,\KK(\CCl(V) \otimes L^2(V))$.

\begin{definition}\label{alpha.tilde}
Define $\tilde{\alpha}_V \colon \U^{\ev}_V \times V \to C^*_{\gr}(\s,\KK(\CCl(V) \otimes L^2(V))$ by $$\tilde{\alpha}_V(A,v) = (-1)^{|A|}A\eta^{\KU_{\RR}}_V(v)A^{-1}.$$  
\end{definition}

\begin{proposition}
The map $\tilde{\alpha}_V$ in Definition \ref{alpha.tilde} induces a $C_2$-equivariant map
$$
\alpha_V \colon \MSpin^c_{\RR}(V) \to \KU_{\RR}(V)
$$
\end{proposition}
\begin{proof}
Let $(A,v) \in \U^{\ev}_V \times V, g \in \Spin^c(V)$, and $s \in \s$. Then 
$$
\begin{aligned}
\tilde{\alpha}_V(Ag^{-1},gv)(s) &=  (-1)^{|A \cdot g^{-1}|}(A\cdot g^{-1})\eta^{\KU_{\RR}}_V(g\cdot v)(s)(A \cdot g^{-1})^{-1} \\
&= (-1)^{|A|}Ai_V(g^{-1})\eta^{\KU_{\RR}}_V(\rho_V(g)v)(s)i_V(g)A^{-1} \\
&= (-1)^{|A|}Ai_V(g^{-1})\eta^{\KU_{\RR}}_V((-1)^{|g|}gvg^{-1})(s)i_V(g)A^{-1}.
\end{aligned}
$$
Now let $w \otimes f \in \CCl(V) \otimes L^2(V)$. For $s = a \in \s$, 
$$
\begin{aligned}
i_V(g^{-1}) \eta^{\KU_{\RR}}_V(gvg^{-1})(a)i_V(g)(w \otimes f) &= i_V(g^{-1}) \eta^{\KU_{\RR}}_V(gvg^{-1})(a)(gw \otimes f\circ \rho_V(g)^{-1}) \\ 
&= i_V(g^{-1}) \bigg( \dfrac{1}{1+|gvg^{-1}|^2} gw \bigg)\otimes P_{{e^{-|\cdot|^2}} }( f\circ \rho_V(g)^{-1}) \\ 
&= g^{-1} \bigg( \dfrac{1}{1+|v|^2} gw \bigg)\otimes P_{{e^{-|\cdot|^2}} }( f) \\ 
&= \bigg( \dfrac{1}{1+|v|^2} w \bigg)\otimes P_{{e^{-|\cdot|^2}} }( f) = \eta^{\KU_{\RR}}_V(v)(a)(w \otimes f),
\end{aligned}
$$
where the third equality follows from isometry invariance of both the integral in the $L^2$-space projection and of the Gaussian. So $\tilde{\alpha}_V(Ag^{-1},gv)(a) = \tilde{\alpha}_V(A,v)(a)$. For $s = b \in \s$, 
$$
\begin{aligned}
i_V(g^{-1}) \eta^{\KU_{\RR}}_V(gvg^{-1})(b)i_V(g)(w \otimes f) &= i_V(g^{-1}) \bigg( \dfrac{gvg^{-1}}{1+|gvg^{-1}|^2} gw \bigg)\otimes P_{{e^{-|\cdot|^2}} }( f\circ \rho_V(g)^{-1}) \\ 
&= g^{-1} \bigg( \dfrac{gvg^{-1}}{1+|v|^2} gw \bigg)\otimes P_{{e^{-|\cdot|^2}} }( f) \\ 
&= \bigg( \dfrac{v}{1+|v|^2} w \bigg)\otimes P_{{e^{-|\cdot|^2}} }( f) = \eta^{\KU_{\RR}}_V(v)(b)(w \otimes f). 
\end{aligned}
$$
Thus, $\tilde{\alpha}_V(Ag^{-1},gv) = \tilde{\alpha}_V(A,v)$, which means $\alpha_V$ is well-defined on $\U^{\ev}_V \times_{\Spin^c(V)} V$. We can extend $\alpha_V$ to $\infty$ by 0, since $\eta^{\KU_{\RR}}_V(\infty) = 0$; thus, $\alpha_V \colon \MSpin^c_{\RR}(V) \to \KU_{\RR}(V)$ is well-defined. For $C_2$-equivariance, we note that $\eta^{\KU_{\RR}}_V$ is $C_2$-equivariant, so we see that 
$$
\overline{\alpha([A,v])} = (-1)^{|A|}\overline{A}\overline{\eta^{\KU_{\RR}}_V(v)}\overline{A}^{-1} = (-1)^{|A|}\overline{A}\eta^{\KU_{\RR}}_V(\overline{v})\overline{A}^{-1} = \alpha_V(\overline{[A,v]}),
$$
for $[A,v] \in \MSpin^c_{\RR}(V)$. 
\end{proof}

\begin{proposition}\label{Real.ABS}
    The maps $\alpha_V$ define a map of commutative orthogonal $C_2$-ring spectra, 
    $$\alpha \colon \MSpin^c_{\RR} \to \KU_{\RR}.$$
\end{proposition}
\begin{proof}
   By the previous proposition we showed that each component, $\alpha_V$, is equivariant. It only remains to check naturality and monoidality, but these follow from Joachim \cite{Joachim}, since our constructions are identical to his on underlying spaces.
\end{proof}

It is shown in \cite{Joachim} that on underlying spectra, $\alpha$ refines the Atiyah--Bott--Shapiro spin${}^c$ orientation of KU \cite{ABS}.

\begin{corollary}
    Taking $C_2$-fixed points, the map $\alpha$ above induces a map of $E_{\infty}$-ring spectra, $$\MSpin \to \KO,$$
    refining the Atiyah--Bott--Shapiro spin orientation of KO \cite{ABS}.
\end{corollary}
\begin{proof}
    The map follows from Proposition \ref{Real.ABS}, Proposition \ref{spin.fixedpoints}, and $(\KU_{\RR})^{C_2} \simeq \KO$. To see that it does indeed refine the Atiyah--Bott--Shapiro orientation, recall that the KO-cohomology class associated to a map $\MSpin(n) \to \KU_{\RR}(\RR^n) $ arises by taking fixed points (= ``real points") in the category of Real $C^*$-algebras (see the proof of Proposition \ref{KRisKR}). The result then follows from the proof in \cite{Joachim}, together with the observation that the spin${}^c$ Atiyah--Bott--Shapiro Thom class associated to a spin vector bundle is the complexification of its spin Thom class.
\end{proof}

This completes the proof of Theorem \ref{maintheorem}.

\bibliographystyle{plain} 
\bibliography{main}

\end{document}